\documentclass[preprint,12pt,3p]{elsarticle}
\usepackage{amssymb}
\usepackage{verbatim}
\usepackage[mathcal]{euscript}
\usepackage[final]{pdfpages}
\usepackage{pdfpages}       
\usepackage{varioref} 
\usepackage{float}
\usepackage{amsfonts,amsmath}
\usepackage{tikz}
\usepackage{epsfig}
\usepackage{epstopdf}
\usetikzlibrary{shadings}
\usetikzlibrary{arrows}
\usetikzlibrary{shapes,snakes}
\usepackage{bm}
\usepackage{mathtools}
\usepackage{graphicx}
\usepackage{caption}
\usepackage[colorlinks=true ]{hyperref}
\usepackage[final]{pdfpages}
\usepackage{multicol}
\usepackage{tikz}
\usepackage{tikz-3dplot} 
\usepackage{pdfpages}       
\usepackage{varioref} 
\usepackage{float}
\usepackage{multirow}
\usepackage{amssymb}
\usetikzlibrary{quotes,arrows.meta}
\numberwithin{equation}{section}

\newtheorem{definition}{Definition}[section]

\newtheorem{theorem}[definition]{Theorem}
\newtheorem{remark}[definition]{Remark}

\newtheorem{proof}[definition]{Proof}
\include{TeXFigs}

\newcommand \bei {\begin{itemize}}
\newcommand \eei {\end{itemize}}

\newcommand \del \partial

\begin{document}
\begin{frontmatter}
 \title{A well-balanced positivity preserving cell-vertex finite volume method satisfying the discrete maximum-minimum principle for coupled models of surface water flow and scalar transport} 
\author{Hasan Karjoun$^a$, Abdelaziz Beljadid$^{a,b,*}$, Philippe G. LeFloch$^c$ }
 \address{$^a$International Water Research Institute, Mohammed VI Polytechnic University, Green City, Morocco}
 \address{$^b$Department of Mathematics and Statistics, University of Ottawa, Ottawa, Canada}
\address{$^c$Laboratoire Jacques-Louis Lions \& Centre National de la Recherche Scientifique, Sorbonne Universit\'e, 4 Place Jussieu, 75258 Paris, France.}

\cortext[mycorrespondingauthor]{Corresponding author.  Email addresses : hassan.karjoun@um6p.ma (H. Karjoun), abeljadi@uottawa.ca (A. Beljadid), contact@philippelefloch.org (P.G. LeFloch)}

\begin{abstract}
We develop a new  finite volume method using unstructured mesh-vertex grids for coupled systems modeling shallow water flows and solute transport over complex bottom topography. Novel well-balanced positivity preserving discretization techniques are proposed for the water surface elevation and the concentration of the pollutant. For the hydrodynamic system, the proposed scheme preserves the steady state of a lake at rest and the positivity of the water depth. For the scalar transport equation, the proposed method guarantees the positivity and a perfect balance of the scalar concentration. The constant-concentration states are preserved in space and time for any hydrodynamic field and complex topography in the absence of source terms of the passive pollutant. Importantly and  this is one of the main features of our approach is that the novel reconstruction techniques proposed for the water surface elevation and concentration satisfy the discrete maximum-minimum principle for the solute concentration. We demonstrate, in a series of numerical tests, the well-balanced and positivity properties of the proposed method and the accuracy of our techniques and their potential advantages in predicting the solutions of the shallow water-transport model.
\end{abstract}

\end{frontmatter}
\section{Introduction}
The shallow water equations(SWEs), also called the Saint-Venant system \citep{de1871theorie}, is a system of partial-differential equations, commonly used to predict and describe flows where the water depth is much smaller than the horizontal length scale of motion and the variations of the 
flow in the vertical direction is negligible compared to its horizontal variations  \citep{vreugdenhil2013numerical,stelling1983construction,beljadid2013unstructured}. The SWEs are widely used in many applications involving free-surface flows in lakes, rivers, and oceans, and become an efficient tool for studying a wide range of hydraulic engineering problems,  as well as  tsunami and weather  predictions   \citep{bonev2018discontinuous,giraldo2002nodal,delestre2013swashes,
CeaVazquez,wu2018well}. The coupled system of SWEs and transport equation is used in many studies to predict the dynamics and the distribution of the concentration of pollutant in flows \citep{CeaVazquez,behzadi2018solution,begnudelli2006unstructured,liang2010well}, which 
have various applications in environmental risk assessment for the development of solutions for sustainable water resources management. The evolution of pollutants is complex, and is subject to various phenomena such as advection, dispersion and diffusion.

In the absence of viscous terms, the SWEs can be considered as a non-linear hyperbolic system of conservation (if the topography is constant) or balance (if the topography is not constant) law. While some analytical solutions are available for the SWEs, they are mainly limited to simple geometries and specific initial and boundary conditions, and in most practical contexts, we need appropriate numerical techniques to approximate the solutions of the system. The design and analysis of numerical schemes for SWEs with different source terms is a very challenging task due to the non-linearity of the system and to the nature of the solutions which can exhibit singularities (shock waves) appearing in finite time even from smooth initial conditions (\citep{leveque2002finite} and references therein). Among the different numerical techniques used to numerically solve the SWEs, the finite volume methods are most convenient to preserve the mass and momentum of the conservative system. Moreover, in the presence of source terms such as variable bottom topography, appropriate numerical techniques should be used to numerically solve the obtained system of balance law to respect a delicate balance between the flux and source terms. The numerical scheme should satisfy the well-balanced property, where the lake at rest steady-state solutions of the system should be preserved, and the numerical scheme should guarantee the positivity of the computed values of the water depth \citep{perthame2001kinetic,dong2020robust,audusse2004fast,hernandez2016central,kurganov2007s,liu2017coupled, ricchiuto2009stabilized}.

Various classes of shock-capturing schemes have been proposed in previous studies to numerically solve the system of SWEs. Godunov-type central schemes \citep{capilla2013new,vcrnjaric2004improved,russo2002central}, and upwind schemes \citep{berthon2012efficient,brufau2003unsteady,cea2010unstructured,leveque1998balancing} using finite volume techniques were proposed for this system.  Upwind schemes are based on approximation techniques of the resolution of Riemann problems at the interfaces of each computational cell,  whereas in central schemes no Riemann solvers are used. Kurganov et al.\cite{kurganov2001semidiscrete,kurganov2005central} introduced central-upwind schemes for hyperbolic conservation laws, where 
the approximation of fluxes at cell interfaces of the control volumes are based on information obtained from the local speeds of non-linear wave propagation. Their main advantages are the high resolution and simplicity of implementation. Central-upwind schemes were used in many recent works to approximate solutions of the SWEs \cite{kurganov2002central,bryson2011well,chertock2015well,hernandez2016central,liu2017coupled,bollermann2013well} and for solving nonlinear hyperbolic conservation laws on curved geometries \citep{beljadid2017central}. Beljadid et al. \cite{beljadid2016well} developed a central-upwind finite volume method on cell-vertex grids for shallow water flow over complex bottom topography, which has an advantage of using more cell interfaces providing more information on the waves propagating in different directions. Their 
techniques enjoy the well-balanced, non-oscillatory and positivity properties and have the advantage to be extended to multidimensional coupled models with unknown eigenstructure where it may be hard to obtain the analytical solution of the Riemann problem or its numerical approximation.

In the present study, we propose a new finite volume method for the simulation and prediction of pollution concentration in water bodies. We consider the coupled system of SWEs and the scalar transport equation \citep{CeaVazquez,behzadi2018solution,wu2018well,begnudelli2006unstructured,liang2010well}, with different source terms, such as variable bottom topography, bottom friction effects and diffusion. Novel techniques are proposed to numerically solve the resulting non-linear system where we extended the methodology proposed in \cite{beljadid2016well} to our system. The proposed method performs well in terms of numerical 
dissipation compared to the original scheme \citep{beljadid2016well, kurganov2007reduction}, especially for the solute transport equation. The developed scheme is well-balanced and preserves the positivity of the computed water depth and the concentration of the 
pollutant in each point of the computational domain at all times. Furthermore, the proposed numerical scheme satisfies the maximum-minimum principle for the concentration \cite{frolkovic1998maximum,kong2013high}, and the constant-concentration state is preserved in space and time for any hydrodynamic field of the flow in the absence of source terms of the passive pollutant.

The paper is organized as follows. In Section \ref{S2}, we present the coupled model for water flow and solute transport system. The proposed numerical scheme is introduced in Section \ref{nouvS3}. In Section \ref{S5}, we present the reconstruction of the hydrodynamic variables and we prove the positivity of
 the water depth for the semi-discrete form of the proposed method. The discretizations of the different source terms are given in Section \ref{S6}. In Section \ref{S7}, a new well-balanced and positivity preserving reconstruction for the scalar concentration is proposed and we prove the maximum and 
minimum principles for the scalar concentration. In Section \ref{S9}, we demonstrate the accuracy and stability of the proposed scheme using a variety of numerical examples. Finally, some concluding remarks are drawn in Section \ref{S10}.
\section{Model equations} \label{S2}
In this study,  we focus on the following coupled model of shallow water flow and solute transport system:
\begin{equation}
\left\{\begin{aligned}
&\frac{\partial h}{\partial t}+\frac{\partial hu}{\partial x}+\frac{\partial hv}{\partial y}=0,\\& \frac{\partial hu }{\partial t}+\frac{\partial }{\partial x}\Big(hu^2+\frac{g}{2}h^2\Big)+\frac{\partial }{\partial y}\Big(huv\Big)=-gh\frac{\partial B}{\partial x}-\frac{\tau_x}{\rho}   ,\\& \frac{\partial hv}{\partial t}+\frac{\partial }{\partial x}\Big(huv\Big)+\frac{\partial }{\partial y}\Big(hv^2+\frac{g}{2}h^2\Big)=-gh\frac{\partial B}{\partial  y}-\frac{\tau_y}{\rho},\\&\frac{\partial hc}{\partial t}+\frac{\partial huc}{\partial x}+\frac{\partial hvc}{\partial y}=\frac{\partial }{\partial x}\left(\gamma h\frac{\partial c}{\partial x}\right)+\frac{\partial }{\partial y}\left(\gamma h\frac{\partial c}{\partial y}\right),
\end{aligned}\right.
\label{Eq1}
\end{equation}
where $h$ is the water depth, $\mathbf{u}:=(u,v)^T$ is the depth averaged velocity field of the flow, $c$ is the concentration of the pollutant, the function $B(x,y)$ represents the bottom elevation, $\rho$ is the water density, $g$ is the gravity acceleration, and $\gamma$ is the effective diffusivity.\\
The components of the friction term are expressed using the Manning formulation as follows:
\begin{equation}
\left\{\begin{aligned}
&\frac{\tau_x}{\rho}=g\frac{n_f^{2}}{h^{1/3}}\left \| \mathbf{u} \right \|u, \\&\frac{\tau_y}{\rho}=g\frac{n_f^{2}}{h^{1/3}}\left \| \mathbf{u} \right \|v,
\end{aligned}\right.
\label{E1}
\end{equation}
where $n_f$ is the Manning coefficient and $\left \| \mathbf{u} \right \|$ is the norm of the vector velocity field of the flow.  

We introduce the new variables of the system: $ s:=hc$ is the conservative variable for the scalar transport equation, $ w:=h+B$ is the water surface  elevation, and $p:=hu$ and $ q:=hv$ are the water discharges in the $x$- and $y$-directions, respectively. With these definitions, the system (\ref{Eq1}) can be 
expressed in the following  form using the vector variable $\bm{U}=(w,p,q,s)^{T}$ for the flux vector and the bottom topography source term:
\begin{equation}
\bm{U}_t+\bm{F}(\bm{U},B)_x+\bm{G}(\bm{U},B)_y=\bm{S}(\bm{U},B)+\bm{I}+\bm{T},
\label{Eq3}
\end{equation}
where the components of the flux vector $\left(\bm{F},\bm{G}\right)^T$, and source terms $\bm{S}$, $\bm{I}$ and $\bm{T}$ are as follows:
\begin{equation}
\begin{aligned}
&\bm{F}(\bm{U},B)=\Big(p,\,\frac{p^2}{w-B}+\frac{g}{2}(w-B)^2,\,\frac{pq}{w-B},\frac{ps}{w-B}\Big)^T,\\
&\bm{G}(\bm{U},B)=\Big(q,\,\frac{pq}{w-B},\,\frac{q^2}{w-B}+\frac{g}{2}(w-B)^2,\frac{qs}{w-B}\Big)^T,\\
&\bm{S}(\bm{U},B)=\left(0,-g(w-B)B_x,-g(w-B)B_y ,0\right)^T,\\
&\bm{I}=\left(0,-g\frac{n_f^{2}}{h^{1/3}}\left \| \mathbf{u} \right \|u, -g\frac{n_f^{2}}{h^{1/3}}\left \| \mathbf{u} \right \|v,0\right)^T,\\
&\bm{T}=\left(0,0,0,\frac{\partial }{\partial x}\left(\gamma h\frac{\partial c}{\partial x}\right)+\frac{\partial }{\partial y}\left(\gamma h\frac{\partial c}{\partial y}\right)\right)^T.
\end{aligned}
\label{Eq4}
\end{equation}

\section{The proposed cell-vertex method}\label{nouvS3}

\subsection{Semi-discrete form of the central-upwind scheme}\label{S3}
In this section, we will first extend the formulation of the cell-vertex central-upwind scheme developed in \cite{beljadid2016well} for our coupled system (\ref{Eq3})-(\ref{Eq4}). We used unstructured cell-vertex 
grids where the domain is partitioned into non-overlapping computational cells $D_j$ of area $|D_{j}|$, obtained by connecting the centroids of the primary triangular grids as shown in Figure~\ref{Fig1}. Let $D_{jk}$ $k=1,2,..., m_j$ be the neighboring cells of $D_j$ with common edges $e_{jk}$ of length $\ell_{jk}$, and define $\bm{n}_{jk}=\left(\cos(\theta_{jk}),\sin(\theta_{jk}) \right)^{T}$ the outward unit normal vector to the cell interface $e_{jk}$. Denote by $(x_j,y_j)$ the coordinates of the center of mass $C_j$ of the cell $D_j$, and $N_{jk}\equiv(x_{m_{jk}},y_{m_{jk}})$ the midpoint of the cell interface $e_{jk}$ having the vertices $N_{jk_{i}}$ $i=1,2$. 
\begin{figure}[ht!]
\centering{}
\begin{tikzpicture}[scale=2.]
\draw [dashed] (-6.1,	0.1) -- (-4.5,	0.1);
\draw [dashed] (-6.1,	0.1) -- (-8.,	0.071);
\draw [dashed] (-4.5,	0.1) -- (-2.5,	0.011);
\draw [dashed] (-6.1,	0.1) -- (-5.25,	1.426584774);
\draw [dashed](-6.1,	0.1) -- (-7.213525492,	0.881677878);
\draw [dashed] (-6.1,	0.1) -- (-7.213525492,	-0.881677878);
\draw [dashed] (-6.1,	0.1) -- (-5.25,	-1.426584774);
\draw [dashed](-4.5,	0.1) -- (-5.25,	1.426584774);
\draw [dashed](-4.5,	0.1) -- (-3.286474508,	0.881677878);
\draw [dashed](-4.5,	0.1) -- (-3.286474508,	-0.881677878);
\draw [dashed](-4.5,	0.1) -- (-5.25,	-1.426584774);
\draw  [dashed] (-5.25,	1.426584774) -- (-7.213525492,	0.881677878)--(-8.,	0.071)--(-7.213525492,	-0.881677878)-- (-5.25,	-1.426584774);
\draw [dashed] (-5.25,	1.426584774)--(-3.286474508,	0.881677878)--(-2.5,	0.011)--(-3.286474508,	-0.881677878)-- (-5.25,	-1.426584774);
\draw [blue]  (-5.283333333,	0.542194925)--(-6.187841831,	0.802754217)--(-7.2,	0.45)--(-7.2,	-0.42)--(-6.187841831,	-0.736087551)--(-5.283333333,	-0.408861591); 
\draw  [blue] (-3.4,	-0.4)--(-3.4,	0.4)--(-4.345491503,	0.802754217)--(-5.283333333,	0.542194925)--(-5.283333333,	-0.408861591)--(-4.345491503,	-0.736087551)--(-3.4,	-0.4);
\node(v0) at ( -6.1,	0.1)  {};
\fill[black] (v0) circle (0.7pt);
\draw ( -6.1,	0.05) node[anchor=north] {$N_{j}$ };

\node(v1) at (  -6.3,	0.0063661 ) {};
\fill [ blue] (v1) circle (0.7pt);
\draw  [ blue] ( -6.5,	0.1) node[anchor=north] {$C_{j}$ };

\node(v2) at ( -5.283333333,	0.066666667 ) {};
\fill[black] (v2) circle (0.7pt);
\draw ( -5.455,	0.15) node[anchor=north] {$N_{jk}$ };

\draw [-stealth] (-5.283333333,	0.25)--(-5.15,	0.25);
\draw (-5.1,0.5) node[anchor=north] {$\bm{n}_{jk}$ };

\node(v3) at (-5.283333333,	0.542194925 ) {};
\fill[black] (v3) circle (0.7pt);
\draw (-5.28,0.9) node[anchor=north] {$N_{jk_{2}}$ };

\node(v4) at (-5.283333333,	-0.408861591) {};
\fill[black] (v4) circle (0.7pt);
\draw (-5.28,-0.44) node[anchor=north] {$N_{jk_{1}}$ };

\node(v5) at (  -4.3,	0.007035752 ) {};
\fill[blue] (v5) circle (0.7pt);
\draw[blue] ( -3.999,0.1) node[anchor=north] {$C_{jk}$ };

\draw[blue]  ( -6.1,	0.6) node[anchor=north] {$D_{j}$ };

\draw[blue]  (-4.4,	0.6) node[anchor=north] {$D_{jk}$ };

\end{tikzpicture}
\caption{Sample of triangular grids (dashed lines)  and control volumes (solid lines).
\label{Fig1}}
\end{figure}
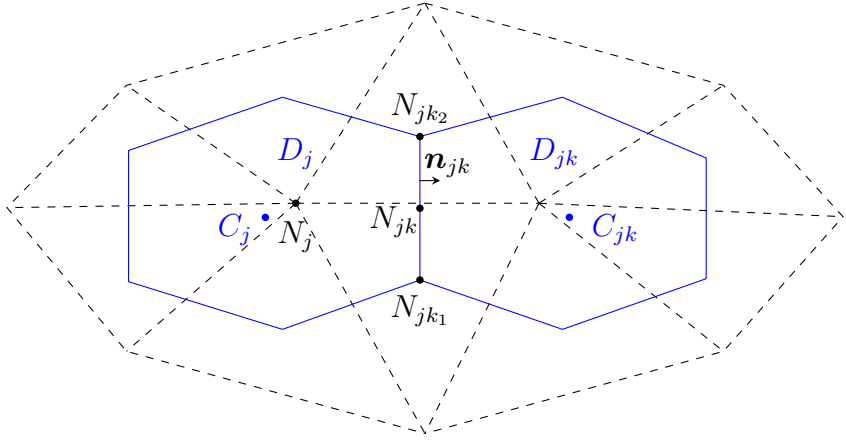

We consider the approximation of the cell average of the computed solution $\bar{\bm{U}}_j$ over the cell-vertex $D_j$:
$$\bar{\bm{U}}_j\approx\frac{1}{|D_j|}\int_{D_j}\bm{U}(x,y,t)\,dxdy.$$

The Jacobian of the system (\ref{Eq3})-(\ref{Eq4}) is:
\begin{equation}
\bm{J}_{jk}:=\dfrac{\partial}{\partial \bm{U}}\left(\left(\bm{F},\bm{G} \right)^T.\bm{n}_{jk}\right) =\cos(\theta_{jk})\dfrac{\partial \bm{F}}{\partial \bm{U}}+\sin(\theta_{jk})\dfrac{\partial \bm{G}}{\partial \bm{U}},
\end{equation}
with 
\begin{equation}
\dfrac{\partial \bm{F}}{\partial \bm{U}}=\left[\begin{array}{c  c c c}
0 \ & 1 \ & 0 \ & 0 \ \\ 
-u^2 +\tilde{c}^2 \ & 2u \ & 0 \ & 0 \ \\ 
-uv \ & v \ & u \ & 0 \ \\ 
-uc \ & c \ & 0 \ & u \
\end{array}  \right], \quad \quad\dfrac{\partial \bm{G}}{\partial \bm{U}}=\left[\begin{array}{c  c c c}
0 \ & 0 \ & 1 \ & 0 \ \\ 
 -uv \ & v \ & u \ & 0 \ \\ 
-v^2 +\tilde{c}^2 \ & v \ & u \ & 0 \ \\ 
-vc \ & 0 \ & c \ & v \
\end{array}  \right].
\end{equation}
This yields to  the following expression of the Jacobian matrix:
\begin{equation}
\bm{J}_{jk}=\left[\begin{array}{c  c c c}
0 \ & n_x \ & n_y \ & 0 \ \\ 
(\tilde{c}^2-u^2)n_x-vn_y\ & 2un_x+vn_y \ & un_y \ & 0 \ \\ 
-uvn_x+(\tilde{c}^2-v^2)n_y\ & vn_x \ & un_x+2vn_y \ & 0 \ \\ 
-ucn_x-vcn_y \ & cn_x \ & cn_y\ & un_x+vn_y \
\end{array}  \right],
\end{equation}
and its eigenvalues are given by:
\begin{equation}
\begin{aligned}
\lambda_{jk}^1=un_x+vn_y-\tilde{c}, \quad \lambda_{jk}^2=\lambda_{jk}^3=un_x+vn_y,\quad
\lambda_{jk}^4=un_x+vn_y+\tilde{c},
\end{aligned}
\end{equation}
where, $\tilde{c}=\sqrt{gh}$, $n_{x}=\cos(\theta_{jk})$, $n_{y}=\sin(\theta_{jk})$ and $u$, $v$ and $h$ are the approximate values of the velocity and water depth at the midpoint of the cell interface $e_{jk}$. 

The cell-vertex central-upwind scheme \cite{beljadid2016well} is applied to the coupled model of water flow and solute transport system (\ref{Eq1}), and its semi-discrete form is written as follows:
\begin{equation}
\begin{aligned}
\frac{d\,\bar{\bm{U}}_j}{dt}=-\frac{1}{|D_j|}\sum_{k=1}^{m_j}\frac{\ell_{jk}}{a^{\rm in}_{jk}+a^{\rm out}_{jk}}
&\biggl[\cos(\theta_{jk}) \left[a^{\rm in}_{jk}\bm{F}(\bm{U}_{jk}(N_{jk}),B_{jk})+a^{\rm out}_{jk}\bm{F}(\bm{U}_j(N_{jk}),B_{jk})\right]\\\\
&+\sin(\theta_{jk}) \left[a^{\rm in}_{jk}\bm{G}(\bm{U}_{jk}(N_{jk}),B_{jk})+a^{\rm out}_{jk}\bm{G}(\bm{U}_{j}(N_{jk}),B_{jk})\right]\\\\
&- a^{\rm in}_{jk}a^{\rm out}_{jk} \left[\bm{U}_{jk}(N_{jk})-\bm{U}_{j}(N_{jk})\right]\biggl]+\bar{\bm{S}}_j+\bar{\bm{I}}_j+\bar{\bm{T}}_j,
\end{aligned}
\quad \label{Eq5}
\end{equation}
where $\bar{\bm{S}}_j\approx\frac{1}{|D_j|}\int_{D_j}\bm{S}(\bm{U}(x,y,t),B(x,y))\,dxdy$, $\bar{\bm{I}}_j\approx\frac{1}{|D_j|}\int_{D_j}\bm{I}(x,y,t)\,dxdy$ and\\
 $\bar{\bm{T}}_j\approx\frac{1}{|D_j|}\int_{D_j}\bm{T}(x,y,t)\,dxdy$ are  the cell averages of the source terms. We denote by $B_{jk}:=B(N_{jk})$ the approximate values of the bottom topography at the midpoints $N_{jk}$, and $\bm{U}_j(N_{jk})$ and $\bm{U}_{jk}(N_{jk})$ correspond respectively to the left and right reconstructed values of the vector variable $\bm{U}$ at the midpoint $N_{jk}$ of the cell interface $e_{jk}$.\\
The extreme values of speeds $a_{jk}^{in}$ and  $a_{jk}^{out}$ are  respectively the absolute values of the smallest and largest eigenvalues of the Jacobian of the system
\begin{equation*}
\begin{aligned}
&a_{jk}^{\rm in}=-\min\big\{\lambda_1[\bm{J}_{jk}(\bm{U}_j(N_{jk}))],\lambda_1[\bm{J}_{jk}(\bm{U}_{jk}(N_{jk}))],0\big\},\\
&a_{jk}^{\rm out}=\max\big\{\lambda_4[\bm{J}_{jk}(\bm{U}_j(N_{jk}))],\lambda_4[\bm{J}_{jk}(\bm{U}_{jk}(N_{jk}))],0\big\},
\end{aligned}
\end{equation*}
which can be expressed as follows:
\begin{equation}
\begin{aligned}
&a_{jk}^{\rm in}=-\min\big\{u_j^{\theta}(N_{jk})-\sqrt{gh_j(N_{jk})},u_{jk}^{\theta}(N_{jk})-\sqrt{gh_{jk}(N_{jk})},0\big\},\\
&a_{jk}^{\rm out}=\max\big\{u_j^{\theta}(N_{jk})+\sqrt{gh_j(N_{jk})},u_{jk}^{\theta}(N_{jk})+\sqrt{gh_{jk}(N_{jk})},0\big\},
\end{aligned}
\label{Eq6}
\end{equation}
where
\begin{equation}
\begin{aligned}
u^\theta_{jk}(N_{jk})&:=\cos(\theta_{jk})u_{jk}(N_{jk})+\sin(\theta_{jk})v_{jk}(N_{jk}),\\
u^\theta_j(N_{jk})&:=\cos(\theta_{jk})u_j(N_{jk})+\sin(\theta_{jk})v_j(N_{jk}).
\end{aligned}
\label{1}
\end{equation}
While the central upwind scheme \citep{beljadid2016well} performs well for modeling SWEs, its extension (\ref{Eq5}) leads to numerical diffusion for the scalar transport equation. As will be shown in our numerical experiments, this is mainly due to the last term $a^{\rm in}_{jk}a^{\rm out}_{jk}\left[\bm{U}_{jk}(N_{jk})-\bm{U}_{j}(N_{jk})\right]$, which has a role to adjust the dissipation in the original scheme but its magnitude leads to numerical diffusion for the scalar concentration for the scheme (\ref{Eq5}).
\subsection{The semi-discrete form of the proposed scheme}\label{S4}
In this section, we introduce the proposed scheme for the coupled SWEs and scalar transport equation. First, we consider the following class of semi-discrete finite volume schemes:   
\begin{equation}
\begin{aligned}
\frac{d\,\bar{\bm{U}}_j^{(i)}}{dt}=-\frac{1}{|D_j|}\sum_{k=1}^{m_j}\frac{\ell_{jk}}{a^{\rm in}_{jk}+a^{\rm out}_{jk}}
&\biggl[\cos(\theta_{jk}) \left[a^{\rm in}_{jk}\bm{F}^{(i)}(\bm{U}_{jk}(N_{jk}),B_{jk})+a^{\rm out}_{jk}\bm{F}^{(i)}(\bm{U}_j(N_{jk}),B_{jk})\right]\\\\
&+\sin(\theta_{jk}) \left[a^{\rm in}_{jk}\bm{G}^{(i)}(\bm{U}_{jk}(N_{jk}),B_{jk})+a^{\rm out}_{jk}\bm{G}^{(i)}(\bm{U}_{j}(N_{jk}),B_{jk})\right]\\\\
&-\delta_{jk} \left[\bm{U}_{jk}^{(i)}(N_{jk})-\bm{U}_{j}^{(i)}(N_{jk})\right]\biggl]+\bar{\bm{S}}_j^{(i)}+\bar{\bm{I}}_j^{(i)}+\bar{\bm{T}}_j^{(i)},\quad i=1, 4,
\end{aligned}
\quad \label{Eq05}
\end{equation}
where $\delta_{jk}$ is a symmetric parameter depending on the variables of our system at the cell $D_{j}$ and the neighboring cell $D_{jk}$ to guarantee the conservation of the method. 

The parameter $\delta_{jk}=a^{\rm in}_{jk}a^{\rm out}_{jk}$ leads to the original central-upwind scheme \citep{beljadid2016well,kurganov2007s,bryson2011well}. Our aim is to propose a numerical scheme (\ref{Eq05}) which has an advantage in reducing 
the numerical dissipation of the original central-upwind scheme \citep{kurganov2007reduction}, especially for the scalar concentration. The main idea is to reduce the last term $a^{\rm in}_{jk}a^{\rm out}_{jk}\left[\bm{U}_{jk}(N_{jk})-\bm{U}_{j}(N_{jk})\right]$ which is responsible for the numerical dissipation of the original central-upwind scheme (\ref{Eq5}) for the scalar transport equation. To this end, we propose a new term $\delta_{jk}\left[\bm{U}_{jk}(N_{jk})-\bm{U}_{j}(N_{jk})\right]$, by introducing a symmetric parameter $\delta_{jk}$ smaller than $a^{\rm in}_{jk}a^{\rm out}_{jk}$ in the semi-discrete form of the proposed scheme. 
While the numerical dissipation decreases by decreasing $\delta_{jk}$, appropriate values will be proposed to avoid numerical oscillations and to ensure the underlying physical properties of our method such as the positivity properties of the water depth and the scalar concentration, and  the discrete maximum-minimum 
principle for the concentration of the pollutant. For this purpose, in our approach, first we consider the same values of the parameter $\delta_{jk}$ for both the continuity equation and the scalar transport equation where $\delta_{jk} \leqslant a^{\rm in}_{jk} a^{\rm out}_{jk}$. According to the methodology developed for central-upwind schemes \citep{beljadid2016well,kurganov2007s,bryson2011well}, to guarantee the positivity of both water depth and concentration of our system, $\delta_{jk}$ should satisfy the following conditions:
\begin{equation*}
\delta_{jk}\geq u_{jk}^{\theta}a^{\rm in}_{jk},\quad \text{and} \quad \delta_{jk}\geq -u_{j}^{\theta}a^{\rm out}_{jk}.
\end{equation*}
 Therefore, the optimum parameter which guarantees the positivity of the water depth and scalar concentration, is given by:
 \begin{equation}
\delta_{jk}^+=\max \left\lbrace u_{jk}^{\theta}a^{\rm in}_{jk},-u_{j}^{\theta}a^{\rm out}_{jk},0\right\rbrace.
\label{dlt}
\end{equation} 
We note that, the selected parameter $\delta_{jk}^+$ can reduce the numerical dissipation of the scheme, but may lead to spurious (unphysical) oscillations. We propose to use a linear combination of $\delta_{jk}^+$ and $a^{\rm in}_{jk}a^{\rm out}_{jk}$ which will be confirmed by our numerical experiments as an appropriate choice for the proposed scheme.
\begin{equation}
\delta_{jk}=\nu \delta_{jk}^+  +(1-\nu)a^{\rm in}_{jk}a^{\rm out}_{jk},
\end{equation}
where $\nu \in [0,1]$. 

 The parameters $\delta_{jk}$ satisfy also the following properties, that we will need to prove both the positivity of water depth and scalar concentration. Following Eq. (\ref{dlt}), we have:
\begin{equation}
\nu\delta_{jk}^+\geq \nu a_{jk}^{\rm in}u^\theta_{jk}(N_{jk})\quad \text{and}\quad \nu \delta_{jk}^+\geq -\nu  a_{jk}^{\rm out}u^\theta_{j}(N_{jk}). 
\label{p1}
\end{equation}
Using Eq. (\ref{Eq6}) related to the extreme values of speeds, we obtain:
 $$
 a^{\rm in}_{jk}a^{\rm out}_{jk}\geq a_{jk}^{\rm in}u^\theta_{jk}(N_{jk}) \quad \text{and}\quad a^{\rm in}_{jk}a^{\rm out}_{jk}\geq -a_{jk}^{\rm out}u^\theta_{j}(N_{jk}), 
$$
then
\begin{equation}
(1-\nu)a^{\rm in}_{jk}a^{\rm out}_{jk}\geq (1-\nu) a_{jk}^{\rm in}u^\theta_{jk}(N_{jk}) \quad \text{and}\quad (1-\nu)a^{\rm in}_{jk}a^{\rm out}_{jk}\geq -(1-\nu)a_{jk}^{\rm out}u^\theta_{j}(N_{jk}). 
\label{p2}
\end{equation}
By combining (\ref{p1}) and (\ref{p2}), we get:
\begin{equation}
 \delta_{jk}\geq a_{jk}^{\rm in}u^\theta_{jk}(N_{jk}),\quad \text{and}  \quad \delta_{jk}\geq- a_{jk}^{\rm out}u^\theta_{j}(N_{jk}).
 \label{delta}
\end{equation}


\section{Reconstruction of the hydrodynamic variables} \label{S5}
\subsection{ Minmod reconstruction}\label{S51}

For the bottom topography, we consider the continuous piecewise linear approximation introduced in \cite{beljadid2016well} where it is assumed that the topography is initially known at the vertices of each cell-vertex. Following the procedure developed in \cite{beljadid2016well}, the approximate values of  the bottom topography $B_{jk}:=B(N_{jk})$ at the midpoints $N_{jk}$ are obtained from the topography at the vertices $N_{jk_{1}}$ and $N_{jk_{2}}$ as shown in Figure \ref{Fig1} by:
\begin{equation}
B_{jk}=\dfrac{B(N_{jk_{1}})+B(N_{jk_{2}})}{2},
\end{equation}
and the topography at the center of mass $C_j$ is  reconstructed by \cite{beljadid2016well}:
\begin{equation}
\quad B_j=\sum_{k=1}^{m_j}\mu_kB_{jk},
\end{equation}
where $\mu_k:=\mathcal A_{jk}/|D_j|$ and $\mathcal A_{jk}$ is the area of the triangle $C_jN_{jk_1}N_{jk_2}$. 

The reconstruction of the topography is the union of $m_{j}$ planes where each plane is defined by the connection of the altitudes of the three points $C_j$, $N_{jk_1}$, and $N_{jk_2}$. For the coupled model of water flow and solute transport (\ref{Eq1}), we extend the procedure developed in \cite{beljadid2016well} to reconstruct the values $\bm{U}_j(N_{jk})$ and $\bm{U}_{jk}(N_{jk})$ at the midpoints of the cell interface $e_{jk}$. The numerical gradient $\nabla \bm{W}_j$ at each cell-vertex is computed using the procedure developed in \cite{beljadid2016well} to obtain the piecewise linear reconstruction for $\bm{W}:=(w,p,q)^T$
\begin{equation}
\bm{W}_j(x,y):=\bar{\bm{W}}_j+(\bm{W}_x)_j(x-x_j)+(\bm{W}_y)_j(y-y_j).
\quad\label{Eq8}
\end{equation}
The reconstruction (\ref{Eq8}) is modified for the  water surface elevation  $w$ by introducing the parameter $\alpha\in [0,1]$ to respect the positivity of the water depth \cite{beljadid2016well}
\begin{equation}
w_j(x,y):=\bar{w}_j+\alpha (\nabla w)_j(x-x_j,y-y_j),
\quad\label{Eq9}
\end{equation}
where we consider the maximum value of the parameter $\alpha$ such that the computed values of the water surface elevation at the cell vertices, obtained from (\ref{Eq9}) satisfy ${w}_{jk_1}\geq B_{jk_1}$. The reconstruction of the water surface elevation and water depth satisfies the following relationships:
\begin{equation}
\begin{aligned}
\bar{w}_j=\sum_{k=1}^{m_j}\mu_kw_j(N_{jk}),\quad\bar{h}_j=\sum_{k=1}^{m_j}\mu_kh_j(N_{jk}).
\end{aligned}
\label{hj}
\end{equation}
\begin{remark}
 The computed values of velocities can be obtained using $u:=p/h$ and $v:=q/h$, respectively. We apply the following desingularization formula \cite{kurganov2007s} to avoid division by very small values of the water depth.
 
\begin{equation}
{u}=\frac{\sqrt{2}\,{h}\,{p}}{\sqrt{{h}^4+max({h}^4,\varepsilon})},\quad {v}=\frac{\sqrt{2}\,{h}\,{q}}{\sqrt{{h}^4+max(h^4,\varepsilon})},
\label{Rg}
\end{equation}
where in our simulations $\epsilon=\max\lbrace \mid D_j\mid^2 \rbrace$.
\end{remark}
The new well-balanced positivity preserving reconstruction developed in this study  for the scalar concentration will be presented in Section \ref{S7}.

\subsection{ The positivity property the water depth}
We apply the forward Euler temporal discretization to the semi-discrete form of the proposed scheme (\ref{Eq05}) to obtain the following explicit form for the water depth:
\begin{equation}
\begin{aligned}
\bar{w}^{n+1}_j=\bar{w}^n_j&-\frac{\Delta t}{|D_j|}\sum_{k=1}^{m_j}\frac{\ell_{jk}\cos(\theta_{jk})}{a^{\rm in}_{jk}+a^{\rm out}_{jk}}
\left[a^{\rm in}_{jk}\left(hu\right)_{jk}(N_{jk})+a^{\rm out}_{jk}\left(hu\right)_j(N_{jk})\right]\\
&-\frac{\Delta t}{|D_j|} \sum_{k=1}^{m_j}\frac{\ell_{jk}\sin(\theta_{jk})}{a^{\rm in}_{jk}+a^{\rm out}_{jk}}
\left[a^{\rm in}_{jk}\left(hv\right)_{jk}(N_{jk})+a^{\rm out}_{jk}\left(hv\right)_j(N_{jk})\right]\\
&+\frac{\Delta t}{|D_j|}\sum_{k=1}^{m_j}\ell_{jk}\frac{\delta_{jk}}{a^{\rm in}_{jk}+a^{\rm out}_{jk}}
\left[w_{jk}(N_{jk})-w_j(N_{jk})\right].
\end{aligned}
\qquad\label{Eq020}
\end{equation}
Due to (\ref{hj}), and since the piecewise linear reconstruction of the bottom topography is continuous, we have $w_{jk}(N_{jk})-w_{j}(N_{jk})=h_{jk}(N_{jk})-h_{j}(N_{jk})$ and Eq.(\ref{Eq020}) can be rewritten in the following form:
\begin{equation}
\begin{aligned}
\bar{h}_j^{n+1}=&\sum_{k=1}^{m_j}h_j(N_{jk})\Big[\mu_k-\frac{\Delta t \ell_{jk}}{|D_j| \left( a^{\rm in}_{jk}+a^{\rm out}_{jk}\right)}
\left[\delta_{jk}+a^{\rm out}_{jk}u^\theta_j(N_{jk})\right]\Big]\\
&+\frac{\Delta t}{|D_j|}\sum_{k=1}^{m_j}h_{jk}(N_{jk})\,\frac{\ell_{jk}}{a_{jk}^{\rm in}+a_{jk}^{\rm out}}
\left[\delta_{jk}-a_{jk}^{\rm in}u^\theta_{jk}(N_{jk})\right],
\end{aligned}
\qquad\label{Eq021}
\end{equation}
where $u^\theta_{j}(N_{jk})$ and $u^\theta_{jk}(N_{jk})$ are given by Eq. (\ref{1}).\\
The positivity is achieved by using $\delta_{jk}\geq a_{jk}^{\rm in} u^\theta_{jk}(N_{jk}) $ as already mentioned in Eq. (\ref{delta}) and $h_{jk}(N_{jk})\ge0$. The last term in (\ref{Eq021}) is positive and since $h_j(N_{jk})\ge0$ and
$$
\frac{\Delta t \ell_{jk}}{|D_j| \left( a^{\rm in}_{jk}+a^{\rm out}_{jk}\right)}
\left[\delta_{jk}+a_{jk}^{\rm out}u^\theta_{j}(N_{jk})\right]\le\frac{\Delta t}{|D_j|}\,\ell_{jk}a^{\rm out}_{jk},
$$
then, the first term will be also positive  under the following condition:
\begin{equation}
\begin{aligned}
\Delta t & \le\frac{\mu_k|D_j|}{\ell_{jk}a^{\rm out}_{jk}}.\\
\end{aligned}
\label{Eq22a}
\end{equation}
This condition is valid under the following time step restriction:
\begin{equation}
\begin{aligned}
\Delta t_h & \le\frac{d_m}{2a},
\end{aligned}
\label{Eq22a}
\end{equation}
where $a:=\max_{j,k}\{a^{\rm in}_{jk},a^{\rm out}_{jk}\}$ and $d_{m}:=\min_{j,k} \lbrace d_{jk}\rbrace$, with $d_{jk}$ is the distance between the center of mass $C_j$ of the cell $D_j$ and its $k$th interface $e_{jk}$.
\section{Discretization  of the source terms}\label{S6}
\subsection{Discretization of the bottom topography and the  friction term} 
The nonzero components of the bottom topography $\bar{\bm{S}}_j$ are approximated using the following well-balanced discretizations \cite{beljadid2016well, bryson2011well}:
\begin{equation}
\begin{aligned}
\bar{S}_j^{(2)}=&\frac{g}{2|D_j|}\sum_{k=1}^{m_j}\ell_{jk}(w_j(N_{jk})-B_{jk})^2\cos(\theta_{jk})-g(w_x)_j(\bar{w}_j-B_j),&\\
\bar{S}_j^{(3)}=&\frac{g}{2|D_j|}\sum_{k=1}^{m_j}\ell_{jk}(w_j(N_{jk})-B_{jk})^2\sin(\theta_{jk})-g(w_y)_j(\bar{w}_j-B_j).&
\end{aligned}
\label{Eq11}
\end{equation}
 For the discretization of the friction source term $\bar{{I}}_j$, we use the following semi-implicit scheme:
\begin{equation}
 \bar{{I}}_j=\frac{1}{|D_j|}\int\limits_{D_j} I(x,y,t)\,dxdy=-\frac{1}{|D_j|}\left(g\frac{n_f^{2}\sqrt{\bar{u}_j^2+\bar{v}_j^2}}{\bar{h}_{j}^{4/3}}\right)^{n}\left( \begin{array}{c}
 0\\ 
\dfrac{\bar{p}_j^{n}+\bar{p}_j^{n+1}}{2} \\ 
 \dfrac{\bar{q}_j^{n}+\bar{q}_j^{n+1}}{2} \\ 
 0
 \end{array} \right).
 \label{I}
\end{equation}
According to (\ref{Rg}), Eq. (\ref{I}) becomes:
\begin{equation}
 \bar{{I}}_j=-\frac{1}{|D_j|}\left(g\frac{2n_f^{2}\bar{h}_{j}^{5/3}\sqrt{\bar{p}_j^2+\bar{q}_j^2}}{\bar{h}_{j}^4+max(\bar{h}_{j}^4,\varepsilon)}\right)^{n}\left( \begin{array}{c}
 0\\ 
 \dfrac{\bar{p}_j^{n}+\bar{p}_j^{n+1}}{2} \\ 
 \dfrac{\bar{q}_j^{n}+\bar{q}_j^{n+1}}{2} \\ 
 0
 \end{array} \right).
\end{equation}
The discretization quadrature (\ref{Eq11}) is well-balanced in the sense that it exactly preserves the steady-state solution of ``\emph{lake at rest}'', since under the conditions $w=constant$, $u=0$ and $v=0$, the friction source term becomes zero $I=0$ and it  has no impact on the well-balanced property. 
\subsection{ Discretization of the diffusion term}
The diffusion term in the scalar transport equation is discretized  using the following  approach  \citep{murillo2006conservative,murillo2008analysis,murillo2005coupling}:
\begin{equation*}
\begin{aligned}
 \bar{{T}}^{(4)}_j&=\frac{1}{|D_j|}\int\limits_{D_j} T^{(4)}(x,y,t)\,dxdy,\\&=\frac{1}{|D_j|}\int\limits_{D_j} \nabla.(\gamma h\nabla c)\,dxdy.
 \end{aligned}
\end{equation*}
By applying the divergence theorem we obtain the following expression for the diffusion term where the notation $(.)_f{_{jk}}$ is used for the estimated values at interfaces:
\begin{equation}\label{D}
\begin{aligned}
 \bar{{T}}^{(4)}_j&=\frac{1}{|D_j|}\sum_{k=1}^{m_j} \ell_{jk}(\gamma h\nabla c)_f{_{jk}}.\bm{n}_{jk},\\
 &=\frac{1}{|D_j|}\sum_{k=1}^{m_j} \ell_{jk} \gamma\tilde{h}_{jk}\left( \dfrac{\bar{c}_{jk}-\bar{c}_j}{\bar{d}_{jk}} \right),
 \end{aligned}
\end{equation}
with $\tilde{h}_{jk}$ is the value of  water depth at the cell interface $e_{jk}$ and $\bar{d}_{jk}$ is the distance  between the projections of the centers of mass $C_{j}$ and $C_{jk}$ on the line with  direction vector $\bm{n}_{jk}$ \citep{murillo2008analysis}, given by:
\begin{equation}
\bar{d}_{jk}= \overrightarrow{C_j C_{jk}}\cdot \bm{n}_{jk}.
\end{equation}
In order to avoid diffusion in wet/dry interfaces the water depth at the cell interface $\tilde{h}_{jk}$ is approximated as \citep{murillo2005coupling}:
\begin{equation}
\tilde{h}_{jk}=\min\lbrace h_j,h_{jk}\rbrace ,
\end{equation}
where $h_j$  and $h_{jk}$ are respectively the left and right reconstructed values of the water depth  at the midpoint $N_{jk}$ of the cell interface $e_{jk}$. 
\section{Well-balanced and positivity of the concentration }\label{S7}
In this section, we propose a new reconstruction for the scalar concentration based on the following three physical properties:  $(i)$ for a passive scalar, constant-concentration state should be preserved in space and time for any 
hydrodynamic field of the flow in the absence of source terms in the scalar transport equation; $(ii)$ the scalar concentration should remain positive at all times; $(iii)$ it should satisfy the discrete maximum-minimum principle.

\subsection{Positivity  preserving  reconstruction for the concentration}
In our approach, the proposed reconstruction will be based on the conservative variable $s=hc$, and we will require the following properties: $(\bm{\pi }_1)$ the reconstruction should satisfy similar relationships $(\ref{hj})$ for the water depth and water surface elevation, 
$\bar{s}_j=\sum_{k=1}^{m_j}\mu_ks_j(N_{jk})$; $(\bm{\pi}_2)$ the reconstruction should satisfy: $\bar{s}_j:= \bar{h}_j \bar{c}_j$ , $s_j:=h_jc_j$ and $s_{jk}:=h_{jk}c_{jk}$; 
$(\bm{\pi}_3)$ for each cell interface $e_{jk}$, the reconstructed values $s_{j}=h_{j}c_{j}$ and $s_{jk}=h_{jk}c_{jk}$ are functions of the cell averages $\bar{c}_{j}$ and $\bar{c}_{j1}$,$\bar{c}_{j2},\ldots,\bar{c}_{jm_j}$ at the cell $D_{j}$ and its  neighboring cells $D_{j1}$, $D_{j2},\ldots,D_{jm_j}$, 
respectively. 
\begin{equation*}
\begin{aligned}
s_j&=\psi\left(\bar{s}_j,\bar{s}_{j1}, \bar{s}_{j2},\ldots,\bar{s}_{jm_j} \right),&\\
&=\psi\left(\bar{c}_j\bar{h}_j,\bar{c}_{j1}\bar{h}_{j1}, \bar{c}_{j2}\bar{h}_{j2},\ldots,\bar{c}_{jm_j}\bar{h}_{jm_j} \right),&
\end{aligned}
\end{equation*}
where the function $\psi$ should satisfy the condition: if $\bar{c}_{j}\equiv\bar{c}_{j1}\equiv\bar{c}_{j2}\equiv\ldots\equiv\bar{c}_{jm_j}\equiv constant$, then $c_j\equiv c_{jk}\equiv constant$, for any hydrodynamic field of the flow. \\
To develop a reconstruction which respects the aforementioned properties, we start by rewriting the conservative variable $s_j$ using the following expressions:
\begin{equation}
\begin{aligned}
s_j&=c_j h_j=\left(\bar{c}_j+\delta_c \right)h_j,&\\
 & =\bar{c}_jh_j+\delta_c h_j=\bar{c}_jh_j+\delta_c \left(\bar{h}_j+\delta_h\right),&\\
& =\bar{c}_jh_j +\delta_c \bar{h}_j+\delta_c  \delta_h,&\\
\end{aligned}
\label{Eq12}
\end{equation}
where $\delta_c$ is computed using the gradient of $c$, $\delta_c =(\nabla c)_j\cdot\overrightarrow{X_{jk}}$, with $\overrightarrow{X_{jk}}=\overrightarrow{C_jN_{jk}}$. \\
A linear reconstruction based on the conservative variable $s$ may cause problem for the well-balanced property of the concentration since the water depth can be nonlinear. The linear reconstruction based on the concentration may cause problem to 
respect the equality of the cell average of the computed solution of the conservative variable $\bar{s}_j=\frac{1}{|D_j|}\int_{D_j}s(x,y,t)\,dxdy$ used in the finite volume method framework. To remedy to these situations, we choose the following reconstruction based on the first two terms of the last equation (\ref{Eq12}):
\begin{equation*}
\begin{aligned}
s_j =\bar{c}_jh_j +\delta_c \bar{h}_j,
\end{aligned}
\label{Eq13}
\end{equation*}
which leads to the following expression for the conservative variable $s$ at the midpoints of cell interfaces:
\begin{equation*}
\begin{aligned}
s_j(N_{jk}) =\bar{c}_jh_j(N_{jk}) +\delta_c \bar{h}_j.
\end{aligned}
\label{Eq14}
\end{equation*}
In order to guarantee the positivity of $s_j$ we choose the largest parameter  $\beta_{\max} \in[0,1]$,  so that $s_j(N_{jk})\geq 0$
\begin{equation}
\begin{aligned}
s_j(N_{jk}) =\bar{c}_jh_j(N_{jk}) +\beta_{\max} \bar{h}_j(\nabla c)_j\cdot \overrightarrow{X_{jk}}.
\end{aligned}
\label{Eq15}
\end{equation}
To compute $(\nabla c)_j$, we apply the same procedure used for $\bm W$ as explained in Section \ref{S51}.
\begin{remark}
For reason of convexity, the positivity of the concentration is achieved throughout computational cells $D_j$, by using Eq. (\ref{Eq15}) at the vertices instead of the midpoints of cell interfaces 
\begin{equation}
\begin{aligned}
s_j(N_{jk_1}) =\bar{c}_jh_j(N_{jk_1}) +\beta_{\max} \bar{h}_j(\nabla c)_j\cdot \overrightarrow{X_{jk_1}},
\end{aligned}
\label{Eq16}
\end{equation}
and imposing the positivity of $s_j(N_{jk_1})$ at all cell vertices to deduce the value of $\beta_{\max}$. 
 \end{remark}
With the proposed  positivity preserving  reconstruction, the properties $\bm{\pi}_1$, $\bm{\pi}_2$, and $\bm{\pi}_3$ can be justified as follows:\\
According to Eq. (\ref{Eq15}), we have:
\begin{equation*}
\begin{aligned}
s_j(N_{jk}) =\bar{c}_jh_j(N_{jk}) +\beta_{\max} \bar{h}_j(\nabla c)_j\cdot \overrightarrow{X_{jk}},
\end{aligned}
\end{equation*}
since
\begin{equation*}
\begin{aligned}
\bar{h}_j=\sum_{k=1}^{m_j}\mu_kh_j(N_{jk}),  \quad \sum_{k=1}^{m_j}\mu_k\overrightarrow{X_{jk}}=0,
\end{aligned}
\end{equation*}
then 
 \begin{equation*}
\begin{aligned}
\sum_{k=1}^{m_j}\mu_ks_j(N_{jk})=\bar{c}_j\bar{h}_j=\bar{s}_j.
\end{aligned}
\end{equation*}
Since we compute the values of the concentration from the values of $s$ and $h$, then our reconstruction satisfies $\bar{s}_j:=\bar{c}_j \bar{h}_j $, $s_j=c_jh_j$ and $s_{jk}=c_{jk}h_{jk}$. \\
If the discrete concentration is constant $\bar{c}_j=c_0$, then $(\nabla c)_j=0$ and Eq. (\ref{Eq15}) reduces to $s_j(N_{jk}) =\bar{c}_jh_j(N_{jk}) = c_0h_j(N_{jk})=c_jh_j(N_{jk})$ and we get $c_j=c_0$. 
\subsection{Well-balanced of the concentration}
Here, we will prove that in the absence of sources terms in the scalar transport equation, the constant-concentration states are preserved in space and time for any hydrodynamic field of the flow. We apply the forward Euler temporal discretization to the semi-discrete form of the proposed scheme (\ref{Eq05}) for the scalar transport equation to obtain the following explicit discretization: 
\begin{equation}
\begin{aligned}
\bar{s}^{n+1}_j=\bar{s}^n_j&-\frac{\Delta t}{|D_j|}\sum_{k=1}^{m_j}\frac{\ell_{jk}\cos(\theta_{jk})}{a^{\rm in}_{jk}+a^{\rm out}_{jk}}
\left[a^{\rm in}_{jk}\left(\frac{ps}{h}\right)_{jk}(N_{jk})+a^{\rm out}_{jk}\left(\frac{ps}{h}\right)_j(N_{jk})\right]\\
&-\frac{\Delta t}{|D_j|}\sum_{k=1}^{m_j}\frac{\ell_{jk}\sin(\theta_{jk})}{a^{\rm in}_{jk}+a^{\rm out}_{jk}}
\left[a^{\rm in}_{jk}\left(\frac{qs}{h}\right)_{jk}(N_{jk})+a^{\rm out}_{jk}\left(\frac{qs}{h}\right)_j(N_{jk})\right]\\
&+\frac{\Delta t}{|D_j|}\sum_{k=1}^{m_j}\ell_{jk}\frac{\delta_{jk}}{a^{\rm in}_{jk}+a^{\rm out}_{jk}}
\left[s_{jk}(N_{jk})-s_j(N_{jk})\right]+ \bar{{T}}^{(4)}_j.
\end{aligned}
\qquad\label{Eq17}
\end{equation}
We assume that $\bar{c}^n=c_0$ at time $t=t^n$, where $c_0$ is a constant.  Under this condition, $\bar{c}_{jk}=\bar{c}_{j}$, in Eq. (\ref{D}), the discrete diffusion term becomes zero $ \bar {{T}}^{(4)}_j = 0 $.\\ 
Finally, by using  the property $\bm{\pi}_3$ we obtain:
\begin{equation}
\begin{aligned}
\bar{s}^{n+1}_j=\bar{s}^n_j&-c_0\biggl[\frac{\Delta t }{|D_j|}\sum_{k=1}^{m_j}\frac{\ell_{jk}\cos(\theta_{jk})}{a^{\rm in}_{jk}+a^{\rm out}_{jk}}
\left[a^{\rm in}_{jk}(hu)_{jk}(N_{jk})+a^{\rm out}_{jk}(hu)_j(N_{jk})\right]\\
&-\frac{\Delta t }{|D_j|}\sum_{k=1}^{m_j}\frac{\ell_{jk}\sin(\theta_{jk})}{a^{\rm in}_{jk}+a^{\rm out}_{jk}}
\left[a^{\rm in}_{jk}(hv)_{jk}(N_{jk})+a^{\rm out}_{jk}(hv)_j(N_{jk})\right]\\&+\frac{\Delta t }{|D_j|}\sum_{k=1}^{m_j}\ell_{jk}\frac{\delta_{jk}}{a^{\rm in}_{jk}+a^{\rm out}_{jk}}
\left[h_{jk}(N_{jk})-h_j(N_{jk})\right]\Big. \dfrac{}{}\biggl],           \\
\quad=c_0\bar{h}&_j^n+c_0\left[\bar{h}_j^{n+1}-\bar{h}_j^{n}\right],\\
\quad=c_0\bar{h}&_j^{n+1}.
\end{aligned}
\qquad\label{Eq19}
\end{equation}
Since $\bar{s} _j ^ {n +1}:= \bar {c} _j ^ {n + 1} \bar{h} _j ^{n +1} = c_0 \bar{h} _j ^ {n + 1} $, then $ \bar{c} _j ^ {n + 1} = c_0 $ and this shows the well-balanced of the concentration for any hydrodynamic
field of the flow.
 \subsection{ The positivity property of the concentration}
In this section we will prove the positivity of the concentration for our system where we consider both the advection and the diffusion effects in the scalar transport equation. The explicit form of the proposed scheme (\ref{Eq05}) applied to the scalar transport equation using the forward Euler temporal discretization yields:
\begin{equation}
\begin{aligned}
\bar{s}^{n+1}_j=\bar{s}^n_j&-\frac{\Delta t }{|D_j|}\sum_{k=1}^{m_j}\frac{\ell_{jk}\cos(\theta_{jk})}{a^{\rm in}_{jk}+a^{\rm out}_{jk}}
\left[ a^{\rm in}_{jk}(us)_{jk}(N_{jk})+a^{\rm out}_{jk}(us)_j(N_{jk})\right]\\
&-\frac{\Delta t }{|D_j|}\sum_{k=1}^{m_j}\frac{\ell_{jk}\sin(\theta_{jk})}{a^{\rm in}_{jk}+a^{\rm out}_{jk}}
\left[a^{\rm in}_{jk}(vs)_{jk}(N_{jk})+a^{\rm out}_{jk}(vs)_j(N_{jk})\right]\\
&+\frac{\Delta t}{|D_j|}\sum_{k=1}^{m_j}\ell_{jk}\frac{\delta_{jk}}{a^{\rm in}_{jk}+a^{\rm out}_{jk}}
\left[s_{jk}(N_{jk})-s_j(N_{jk})\right]
\\
&+
\frac{\Delta t}{|D_j|}\sum_{k=1}^{m_j} \ell_{jk} \gamma\tilde{h}_{jk} \left( \dfrac{\bar{c}_{jk}-\bar{c}_j}{\bar{d}_{jk}} \right),
\end{aligned}
\qquad\label{Eq20}
\end{equation}
which can be rewritten as:
\begin{equation}
\begin{aligned}
\bar{s}^{n+1}_j=\bar{s}^n_j & -\frac{\Delta t}{|D_j|}\sum_{k=1}^{m_j} \ell_{jk} \gamma\tilde{h}_{jk} \dfrac{\bar{c}_j}{\bar{d}_{jk}} -\frac{\Delta t}{|D_j|}\sum_{k=1}^{m_j}s_j(N_{jk})\frac{ \ell_{jk}}{ \left( a^{\rm in}_{jk}+a^{\rm out}_{jk}\right)}
\left[\delta_{jk}+a^{\rm out}_{jk}u^\theta_j(N_{jk})\right]\\
&+\frac{\Delta t}{|D_j|}\sum_{k=1}^{m_j}\ell_{jk}\Big[\frac{s_{jk}(N_{jk})}{a_{jk}^{\rm in}+a_{jk}^{\rm out}}
\left[\delta_{jk}-a_{jk}^{\rm in}u^\theta_{jk}(N_{jk})\right]+  \gamma\tilde{h}_{jk}  \dfrac{\bar{c}_{jk}}{\bar{d}_{jk}}\Big].
\end{aligned}
\qquad\label{E21}
\end{equation}
We have, $\tilde{h}_{jk}\bar{c}_j=\min\lbrace h_j,h_{jk}\rbrace\bar{c}_j \leq h_j\bar{c}_j, \quad \text{and} \quad \bar{h}_j=\sum_{k=1}^{m_j}\mu_k h_j(N_{jk})\geq \mu_k h_j(N_{jk}), $ for any $k$, $1\leq k\leq m_j$, then, $\tilde{h}_{jk}\bar{c}_j\leq\mu_k^{-1}\bar{h}_j\bar{c}_j=\mu_k^{-1}\bar{s}_j.$\\ According to Eq.(\ref{E21}) we deduce:
\begin{equation}
\begin{aligned}
\bar{s}^{n+1}_j\geq \bar{s}^n_j &\left(1 - {\Delta t}\lambda_j \gamma \right) -\frac{\Delta t}{|D_j|}\sum_{k=1}^{m_j}s_j(N_{jk})\frac{ \ell_{jk}}{ \left( a^{\rm in}_{jk}+a^{\rm out}_{jk}\right)}
\left[\delta_{jk}+a^{\rm out}_{jk}u^\theta_j(N_{jk})\right]\\
&+\frac{\Delta t}{|D_j|}\sum_{k=1}^{m_j}\ell_{jk}\Big[\frac{s_{jk}(N_{jk})}{a_{jk}^{\rm in}+a_{jk}^{\rm out}}
\left[\delta_{jk}-a_{jk}^{\rm in}u^\theta_{jk}(N_{jk})\right]+  \gamma\tilde{h}_{jk}  \dfrac{\bar{c}_{jk}}{\bar{d}_{jk}}\Big],
\end{aligned}
\qquad\
\end{equation}
where the geometrical 
parameter $\lambda_j$ is given by:
\begin{equation}
\lambda_j=\frac{1}{|D_j|}\sum_{k=1}^{m_j}\mu_k^{-1} \dfrac{ \ell_{jk} }{\bar{d}_{jk}} 
=\sum_{k=1}^{m_j}\frac{1}{d _{jk}\bar{d}_{jk}}.
\label{Eqlamda}
\end{equation}

 Due to $\bm{\pi}_1 $, we use the relationship $\bar{s}_j^n=\sum_{k=1}^{m_j}\mu_ks_j(N_{jk})$  to get the following inequality:
\begin{equation}
\begin{aligned}
\bar{s}^{n+1}_j \geq &\sum_{k=1}^{m_j}s_j(N_{jk}) \Big[\mu_k\left(1 - {\Delta t}\lambda_j\gamma \right) -\frac{ \Delta t\ell_{jk}}{|D_j| \left( a^{\rm in}_{jk}+a^{\rm out}_{jk}\right)}
\left[\delta_{jk}+a^{\rm out}_{jk}u^\theta_j(N_{jk})\right]\Big]\\
&+\frac{\Delta t}{|D_j|}\sum_{k=1}^{m_j}\ell_{jk}\Big[\frac{s_{jk}(N_{jk})}{a_{jk}^{\rm in}+a_{jk}^{\rm out}}
\left[\delta_{jk}-a_{jk}^{\rm in}u^\theta_{jk}(N_{jk})\right]+  \gamma\tilde{h}_{jk}  \dfrac{\bar{c}_{jk}}{\bar{d}_{jk}}\Big].
\end{aligned}
\qquad\
\label{Eq21}
\end{equation}
The last term of (\ref{Eq21}) is non-negative since $\delta_{jk}\geq a_{jk}^{\rm in} u^\theta_{jk}(N_{jk}) $ as shown in Eq.(\ref{delta}), $s_{jk}(N_{jk})\ge0$ and $\bar{c}_{jk}\ge0$. The first term on the right-hand side of (\ref{Eq21}) will be non-negative under a time step restriction to be determined using the condition
\begin{equation}
\mu_k-\Delta t \left( \mu_k\lambda_j\gamma +\frac{\ell_{jk}}{|D_j| \left( a^{\rm in}_{jk}+a^{\rm out}_{jk}\right)}
\left[\delta_{jk}+a_{jk}^{\rm out}u^\theta_{j}(N_{jk})\right] \right)\geq 0,
\label{E22}
\end{equation}
which is satisfied by requiring the following inequality:
\begin{equation}
\mu_k-\Delta t \left( \mu_k\lambda_j\gamma +\frac{\ell_{jk}}{|D_j|}
a^{\rm out}_{jk} \right)\geq 0.
\label{E23}
\end{equation}
The condition (\ref{E23}) is simply:
\begin{equation}
\Delta t\leq \frac{d_{jk}}{d_{jk}\lambda_{j} \gamma+2a^{\rm out}_{jk}}.
\label{E23b}
\end{equation}
It is straightforward to show that the positivity of the conservative variable $s$ is guaranteed by imposing the following time step restriction:
\begin{equation}
\begin{aligned}
\Delta t_c & \le \frac{d_{m}}{d_{m}\lambda_{M} \gamma+2a},\\
\end{aligned}
\label{Eq22a}
\end{equation}
where $\lambda_{M}:=\max_j \lbrace\lambda_{j}\rbrace$, and as already defined $d_{m}=\min_{j,k} \lbrace d_{jk}\rbrace$ and $a=\max_{j,k}\lbrace a^{out}_{jk},a^{in}_{jk} \rbrace.$

\subsection{ The time step restriction of the proposed  scheme}
The stability condition of the proposed scheme using the explicit Euler method in the temporal discretization is determined by the combination of advection and diffusion where the time step is limited by both the Courant-Friedrichs-Lewy number, $CFL$, and the Peclet number, $P_e$ \citep{li2012fully,murillo2005coupling}, as follows:
\begin{equation}
CFL+P_e\leq 1,
\end{equation}
The extreme local speeds at interfaces used in the proposed scheme are the same as those used for cell-vertex central upwind scheme for shallow water system \citep{beljadid2016well}. The time step restriction of this scheme can be applied for our case for the advection where the following CFL number is used
\citep{beljadid2016well}:
\begin{equation}
CFL=\frac{a}{\min_{j,k}\{d_{jk}\}}\Delta t_s.
\label{CFLhd}
\end{equation}
The von-Neumann stability analysis of the proposed discretization of the diffusion term leads to a time step restriction using the following Peclet number:
\begin{equation}
\begin{aligned}
P_e=&\gamma\sum_{k=1}^{m_j}\frac{2\mu_k}{d _{jk}\bar{d}_{jk}}\Delta t_s,\\
=&\gamma \tilde{\lambda}_j \Delta t_s,
\end{aligned}
\label{pe}
\end{equation}
where we used the following geometrical parameter of the same order as the parameter $\lambda_j$ defined in (\ref{Eqlamda}) in the condition developed for the positivity of the proposed scheme:

\begin{equation}
\tilde{\lambda}_j= \sum_{k=1}^{m_j}\frac{2\mu_k}{d _{jk}\bar{d}_{jk}}.
\label{Eqlamdp}
\end{equation}
Then, the stability requirement of the proposed numerical scheme is given by the following time step restriction:
\begin{equation}
\Delta t_{s}\leq \frac{d_m}{d_m\tilde{\lambda}_M\gamma+a},
\label{dts}
\end{equation}
where $\tilde{\lambda}_{M}:=\max_j \lbrace\tilde{\lambda}_{j}\rbrace$.
\begin{remark}
The terms on the right-hand sides in Eqs.(\ref{Eq22a}) and (\ref{dts}) for the time step restrictions respectively for the positivity and stability of the proposed scheme, have the same order thanks to the expressions of the geometrical parameters $\lambda_M$ and $\tilde{\lambda}_M$.
Under the conditions (\ref{Eq22a}) and (\ref{dts}), the positivity of the water depth and the concentration can be preserved if we use the three-stage third-order SSP Runge-Kutta method as temporal scheme since it is a convex combination of three forward Euler methods.
\end{remark}

\subsection{Maximum and minimum principles for the concentration}\label{S8}
In this section, we will prove the maximum and minimum principles for the scalar concentration. First, we will complement our reconstruction of the conservative variable $s=hc$. In the proposed procedure of the reconstruction, the parameter $\beta_{\max}$ used in Eq. (\ref{Eq16}) should also satisfy:
\begin{equation}
\min(\bar{c_j},\bar{c}_{jk})\leq c_j(N_{jk})\leq \max(\bar{c_j},\bar{c}_{jk}).
\label{Eq23}
\end{equation}
To avoid division by very small values of the water depth in the computation of the cell average of the concentration, we propose the following formula:
\begin{equation}
\bar{c}_{j} =\left\{ 
\begin{array}{ll}
\dfrac{\bar{s}_{j}}{\bar{h}_{j}}, & \text{if }\quad\bar{h}_{j}>\xi , \\ 
\min\left\lbrace c_{max},\max\left\lbrace c_{min},R_{j} \right\rbrace \right\rbrace,& \text{otherwise,} %
\end{array}%
\right.
\label{Eq23b}
\end{equation}%
where $c_{max}$ and $c_{min}$ are respectively the maximum and the minimum values of the concentration $\bar{c}_{j}$ and the concentrations $\bar{c}_{jk}$ at the neighboring cells,  in the previous step and $R_j$ are obtained from the desingularization formula:
$$
R_j=\frac{\sqrt{2}\,\bar{h}_{j}\,\bar{s}_{j}}{\sqrt{\bar{h}_{j}^4+max(\bar{h}_{j}^4,\varepsilon})},
$$
where we use the same $\varepsilon$ given in (\ref{Rg}) and  $\xi=10^{-6}$.

In the following, we will prove that the proposed scheme satisfies the maximum and minimum principles for the scalar concentration.
\begin{theorem}
 The semi-discrete form of the proposed scheme (\ref{Eq05}) for the scalar transport equation  without diffusion term, together with the proposed reconstruction (\ref{Eq15}) with the properties  (\ref{Eq23}-\ref{Eq23b}), ensure the maximum and the minimum principles for the 
scalar concentration:
\begin{equation}
 \min (\bar{c}_{j}^{n},\bar{c}_{j1}^{n},\bar{c}_{j2}^{n}\ldots,\bar{c}_{jm_j}^{n}) \leq \bar{c}_j^{n+1}\leq \max (\bar{c}_{j}^{n},\bar{c}_{j1}^{n},\bar{c}_{j2}^{n}\ldots,\bar{c}_{jm_j}^{n}).
\end{equation}
\end{theorem} 
 \begin{proof}
 Let $c_{\min}=\min(\bar{c}_{j}^{n},\bar{c}_{j1}^{n},\bar{c}_{j2}^{n}\ldots,\bar{c}_{jm_j}^{n})$ and $c_{\max}=\max (\bar{c}_{j}^{n},\bar{c}_{j1}^{n},\bar{c}_{j2}^{n}\ldots,\bar{c}_{jm_j}^{n})$. Following Eq. (\ref{Eq23}), we have $c_{\min}\leq c_j(N_{jk})\leq c_{\max}$ and by combining these inequalities with Eq. (\ref{Eq20}) with $\bar{T}_j^{(4)}=0$ and using $s_j(N_{jk})=c_j(N_{jk})h_j(N_{jk})$ and $s_{jk}(N_{jk})=c_{jk}(N_{jk})h_{jk}(N_{jk})$, we obtain:
\begin{equation}
\begin{aligned}
\bar{s}_j^{n+1}\leq c_{\max} &\left\lbrace \sum_{k=1}^{m_j}h_j(N_{jk})\Big[\mu_k-\frac{\Delta t \ell_{jk}}{|D_j| \left( a^{\rm in}_{jk}+a^{\rm out}_{jk}\right)} \left[\delta_{jk}+a^{\rm out}_{jk}u^\theta_j(N_{jk})\right]\Big]\right.\\
&\left.+\frac{\Delta t}{|D_j|}\sum_{k=1}^{m_j}h_{jk}(N_{jk})\,\frac{\ell_{jk}}{a_{jk}^{\rm in}+a_{jk}^{\rm out}}
\left[\delta_{jk}-a_{jk}^{\rm in}u^\theta_{jk}(N_{jk})\right]\right\rbrace,\\
\leq c_{\max}& \left\lbrace \sum_{k=1}^{m_j}h_j(N_{jk})\mu_k-\frac{\Delta t}{|D_j|}\sum_{k=1}^{m_j}\frac{\ell_{jk}}{a^{\rm in}_{jk}+a^{\rm out}_{jk}} \left[\delta_{jk}+a^{\rm out}_{jk}u^\theta_j(N_{jk})\right]\right.\\
&\left.+\frac{\Delta t}{|D_j|}\sum_{k=1}^{m_j}h_{jk}(N_{jk})\,\frac{\ell_{jk}}{a_{jk}^{\rm in}+a_{jk}^{\rm out}}
\left[\delta_{jk}-a_{jk}^{\rm in}u^\theta_{jk}(N_{jk})\right]\right\rbrace,
\end{aligned}
\qquad\label{Eq24}
\end{equation} 
and
\begin{equation}
\begin{aligned}
\bar{s}_j^{n+1}\geq c_{\min} &\left\lbrace \sum_{k=1}^{m_j}h_j(N_{jk})\Big[\mu_k-\frac{\Delta t \ell_{jk}}{|D_j| \left( a^{\rm in}_{jk}+a^{\rm out}_{jk}\right)} \left[\delta_{jk}+a^{\rm out}_{jk}u^\theta_j(N_{jk})\right]\Big]\right.\\
&\left.+\frac{\Delta t}{|D_j|}\sum_{k=1}^{m_j}h_{jk}(N_{jk})\,\frac{\ell_{jk}}{a_{jk}^{\rm in}+a_{jk}^{\rm out}}
\left[\delta_{jk}-a_{jk}^{\rm in}u^\theta_{jk}(N_{jk})\right]\right\rbrace,\\
\geq c_{\min}& \left\lbrace \sum_{k=1}^{m_j}h_j(N_{jk})\mu_k-\frac{\Delta t}{|D_j|}\sum_{k=1}^{m_j}\frac{\ell_{jk}}{a^{\rm in}_{jk}+a^{\rm out}_{jk}} \left[\delta_{jk}+a^{\rm out}_{jk}u^\theta_j(N_{jk})\right]\right.\\
&\left.+\frac{\Delta t}{|D_j|}\sum_{k=1}^{m_j}h_{jk}(N_{jk})\,\frac{\ell_{jk}}{a_{jk}^{\rm in}+a_{jk}^{\rm out}}
\left[\delta_{jk}-a_{jk}^{\rm in}u^\theta_{jk}(N_{jk})\right]\right\rbrace.
\end{aligned}
\qquad\label{Eq224}
\end{equation} 
 Using  $\sum_{k=1}^{m_j} \mu_k h_j(N_{jk})=\bar{h}_j^{n}$, we obtain:
\begin{equation}
\begin{aligned}
 c_{\min} \bar{h}_j^{n+1}\leq\bar{s}_j^{n+1}&\leq c_{\max} \bar{h}_j^{n+1}.
\end{aligned}
\qquad\label{Eq25}
\end{equation} 
Since, $\bar{s}_j^{n+1}=  \bar{c}_j^{n+1}\bar{h}_j^{n+1}$ we conclude that $c_{\min}\leq\bar{c}_j^{n+1}\leq c_{\max}$ and this completes the proof.
 \end{proof}

\section{Numerical experiments}\label{S9}
In this section, the proposed scheme for the coupled model for water flow and solute transport system (\ref{Eq1}), is validated against several benchmark tests. In all of the numerical experiments, we have set $g=9.81$ and we used different values of the parameter $\nu$. In the numerical example \ref{T1}, a dam-break problem with constant concentration is considered to test the ability of the proposed scheme to preserve constant state of concentration for any hydrodynamic where we consider rapidly varying flow. The numerical tests \ref{T2} and \ref{T4} are performed for convection-diffusion problems. In the numerical example \ref{T5}, a dam-break problem over three humps is used to assess the model's ability for the prediction of the flow over complex bottom topography with scalar concentration, involving wet/dry areas.
\subsection{Example 1: Dam break with scalar transport}\label{T1}
We start our numerical experiments with the following example, where wet and dry bed dam-break problems are considered. A computational domain $[0,10]\times[0,5]$ with frictionless bed is discretized  using an average cell area $\left |\bar{D_{j}} \right |=4.962 \times 10^{-3}$. The dam is located at $x_0=10$, and the initial water depth is set to $h_l=1$ upstream of the dam, whereas downstream the depth is $h_r=0.1$ for wet bed dam and $h_r=0$ for dry bed dam, with zero velocity field. We perform the numerical simulations using outflow boundary conditions and we compare the results of the proposed numerical method with the analytical solutions for our problem which are described in \cite{delestre2013swashes}. 

Figure \ref{Fig2} shows the profile of the computed water depth, compared with the analytical solution in dry dam-break and in wet dam-break at time $t=1$. To show the well-balanced property of the proposed scheme for the concentration, we consider an initial condition with a constant concentration over the domain. The results of our numerical simulations demonstrate that the well-balanced property is satisfied where the constant state of the concentration is preserved in the whole domain for any hydrodynamic field of the flow.
\begin{figure}[h!]
\begin{center}
(a)\\ 
\includegraphics[scale=0.32]{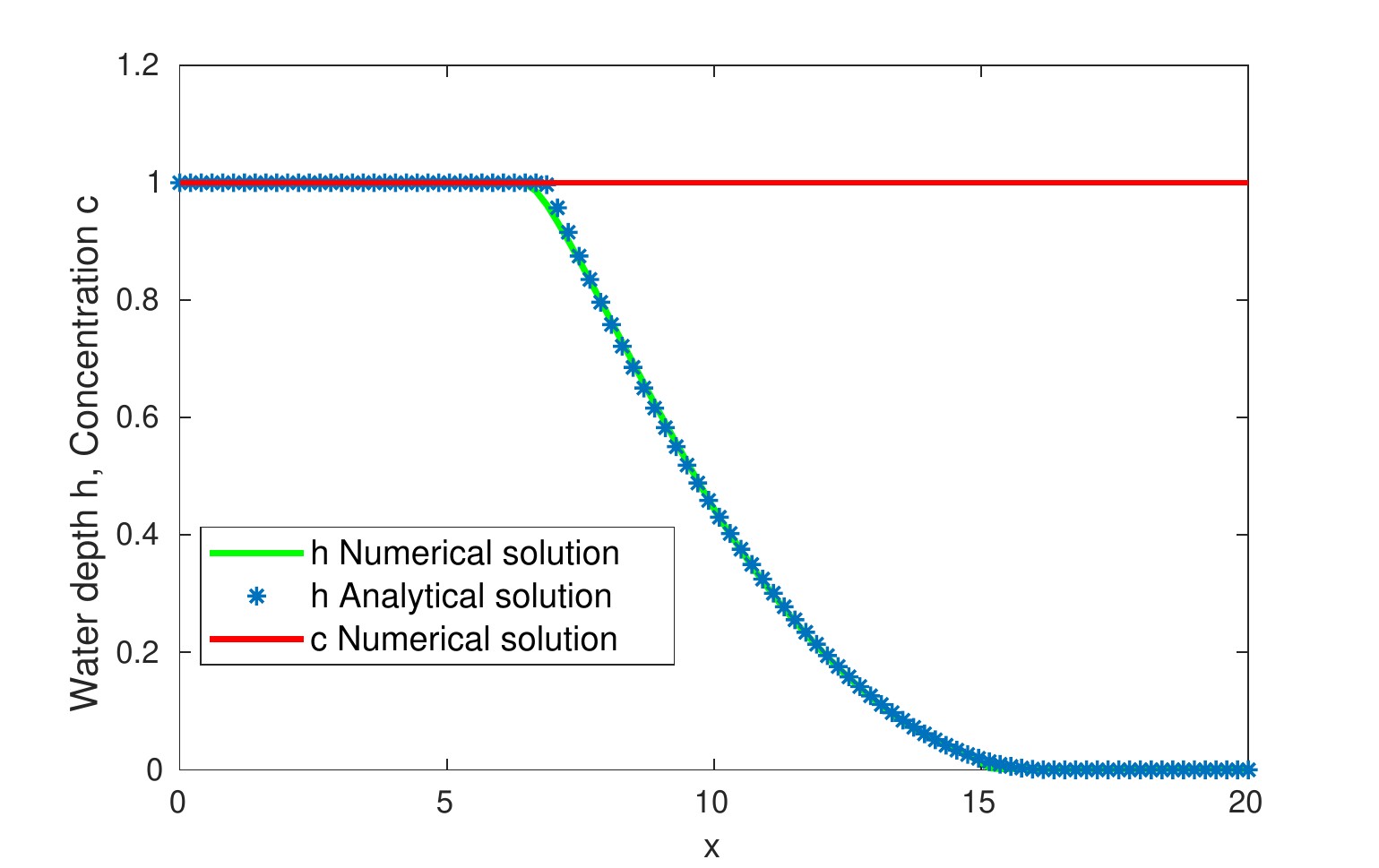}\quad \includegraphics[scale=0.25]{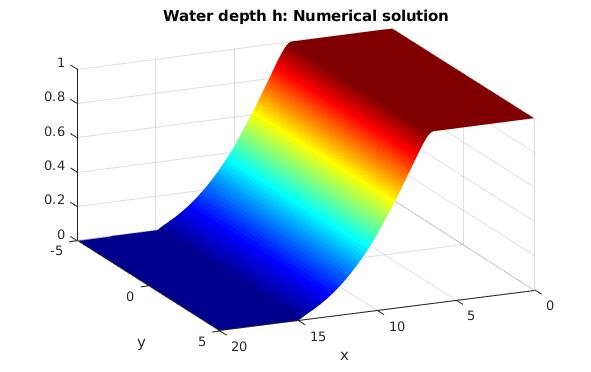}\quad \includegraphics[scale=0.42]{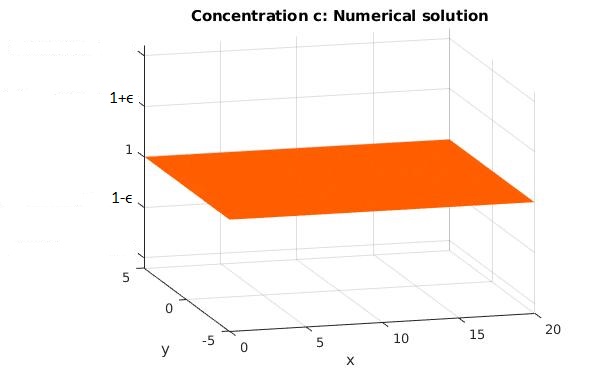} 
(b)\\ 
\includegraphics[scale=0.32]{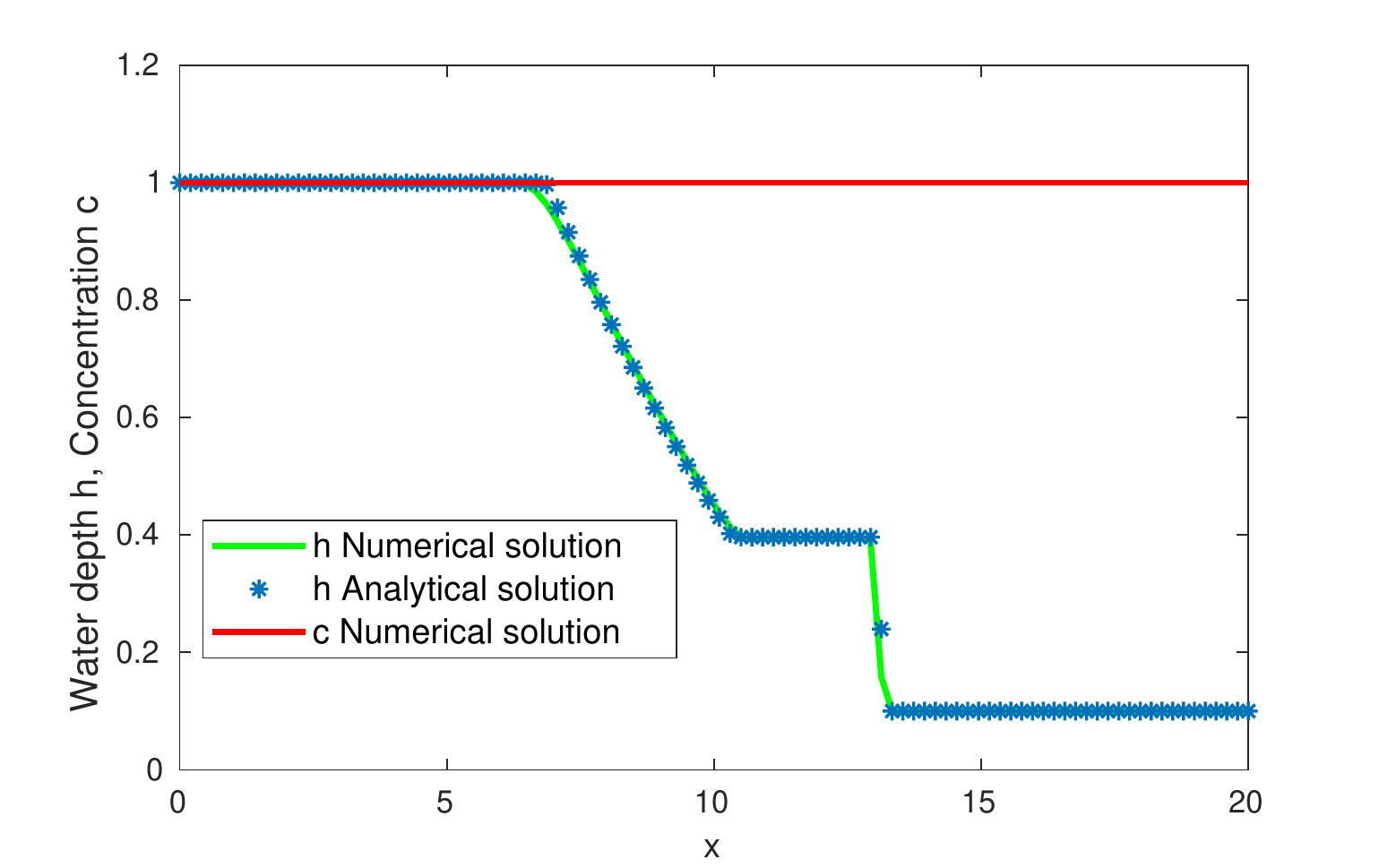}\quad \includegraphics[scale=0.25]{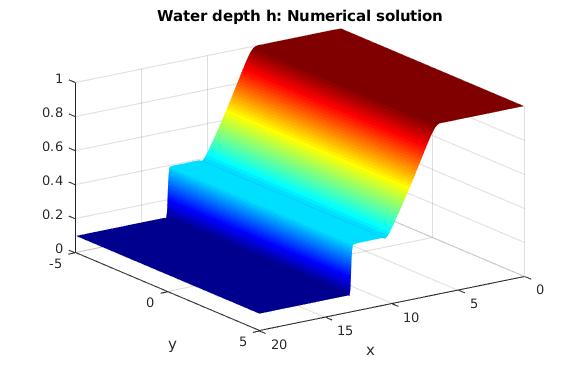}\quad \includegraphics[scale=0.42]{2Ddamcc1.jpg} 
\caption{Numerical and analytical solutions for the water depth and concentration at time $t=1$ for dry dam-break (a) and wet dam-break (b) problems. The computed solutions of the concentration are shown with $\epsilon=10^{-15}$.  }
\label{Fig2}
\end{center}
\end{figure}
\subsection{Example 2: Passive scalar advection }\label{T2}
In this numerical example, we consider the advection of an amount of pollutant by shallow water flows over flat and smooth bed. The  computational domain is $[0,10]\times[0,10]$ which is discretized  using an average cell area $\left |\bar{D_{j}} \right |=1.243\times 10^{-3}$. The water depth is initially constant 
everywhere with $h=1$, the velocity field of the flow is uniform $u=v=0.5$, and the concentration is given by:
\begin{equation}
c(x,y,0)=\left\{ 
\begin{array}{ll}
1, & \text{if }\quad  \left \|  \bm x-\bm x_0\right \|  \leq 1 , \\ 
0,& \quad\quad\text{otherwise} %
\end{array}%
\right.
\end{equation}%
where, $\bm x=(x,y)$ and $\bm x_0=(1.5,1.5)$.

In our numerical simulations, inflow and outflow conditions are applied at the boundaries. Figure \ref{Fig22} (right) shows the two-dimensional evolution of the computed scalar concentration at different times where the advection of the passive concentration follows 
the motion of the flow and moves diagonally across the domain with the constant speed of the flow. In Figure \ref{Fig22} (left), we present the cross section along the $x$-axis of the numerical solutions for both the proposed scheme and the original central-upwind scheme compared with the analytical solution. Our 
results show that the proposed method performs well in terms of numerical diffusion compared to the original central-upwind scheme \cite{beljadid2016well} and the numerical solution can be further improved by using a refined mesh.
\begin{figure}
\begin{center}
\includegraphics[scale=0.35]{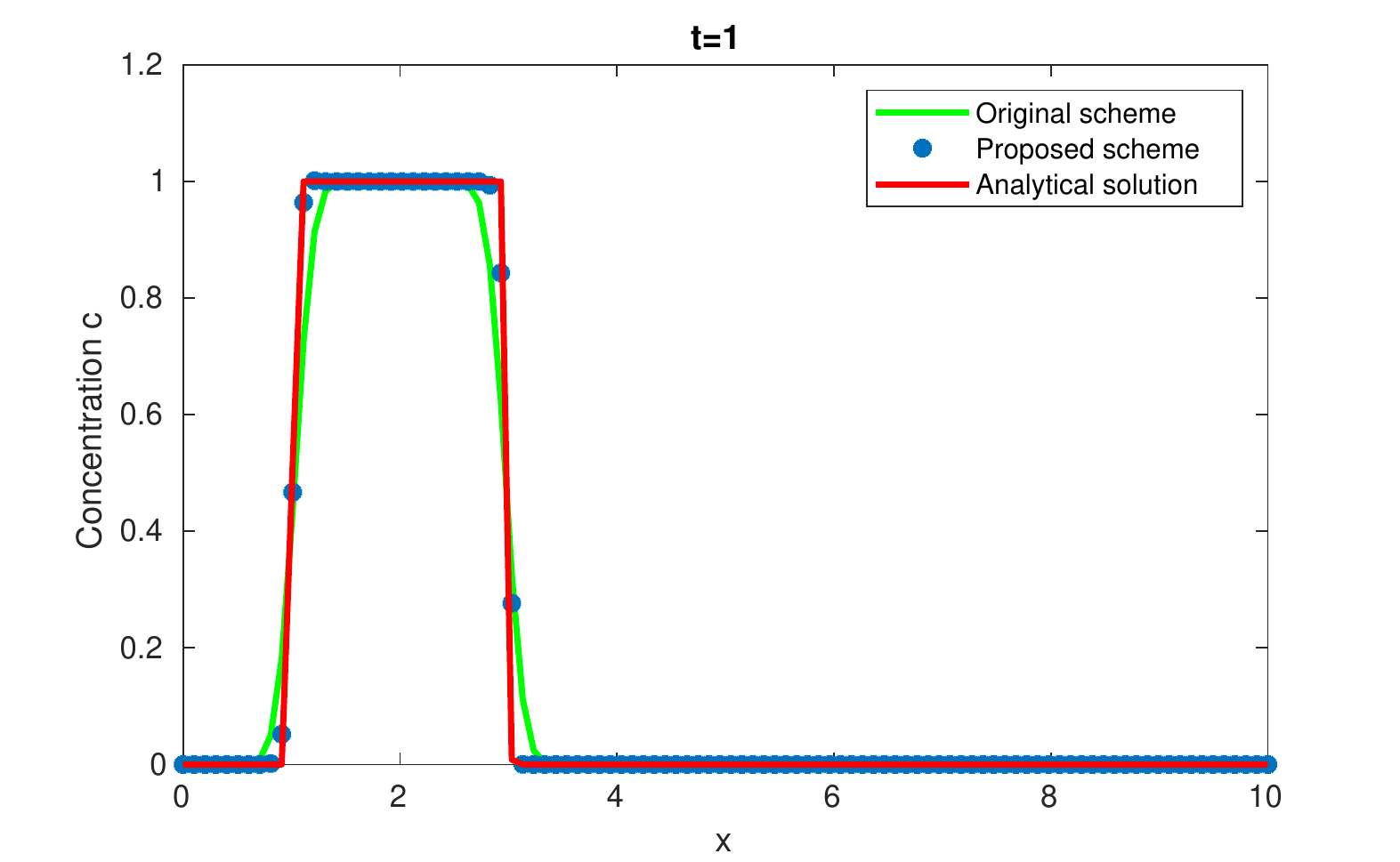} \quad \includegraphics[scale=0.25] {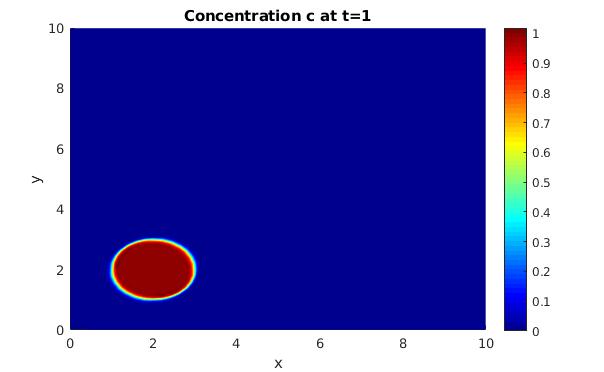}
\includegraphics[scale=0.35]{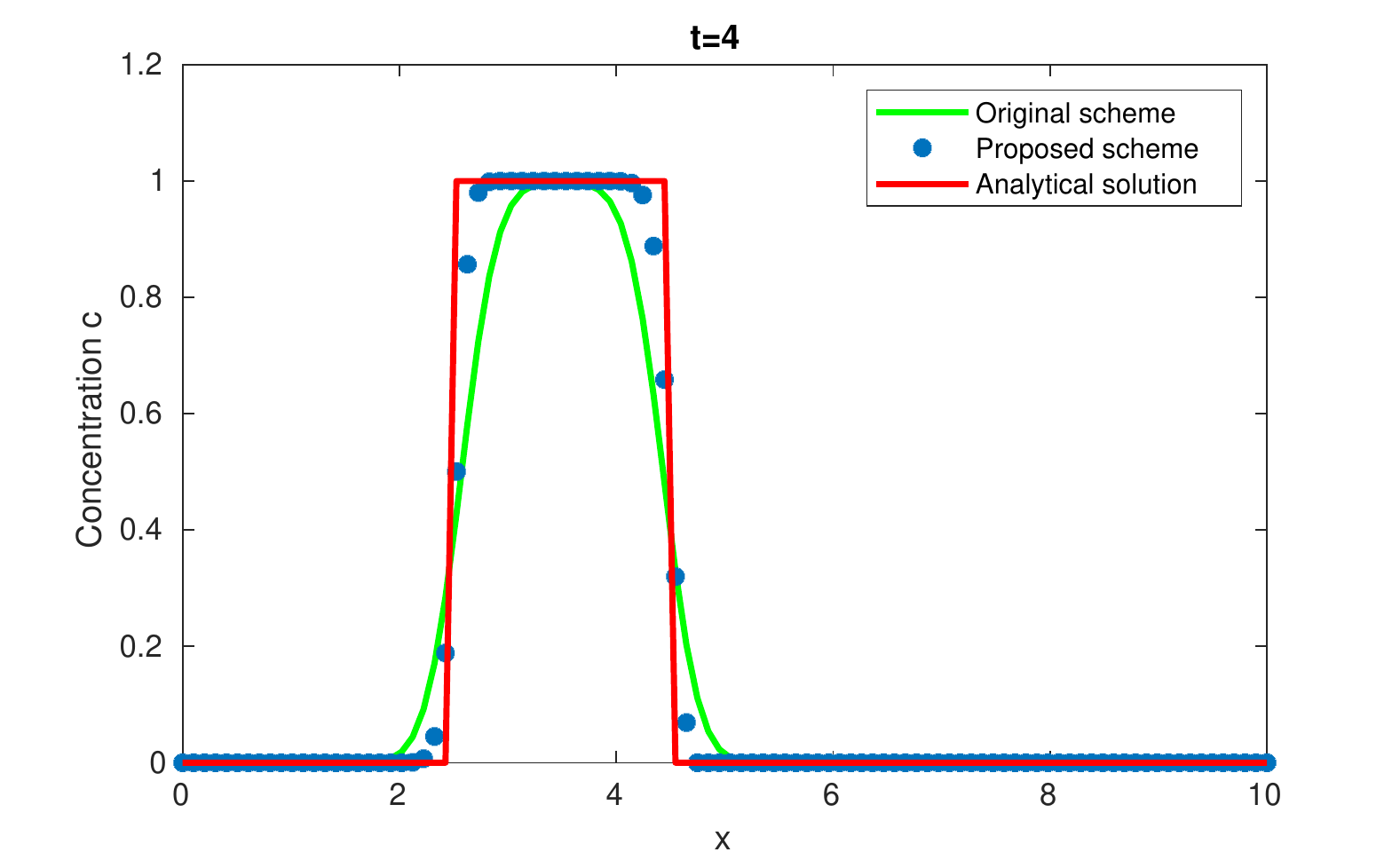} \quad \includegraphics[scale=0.25]{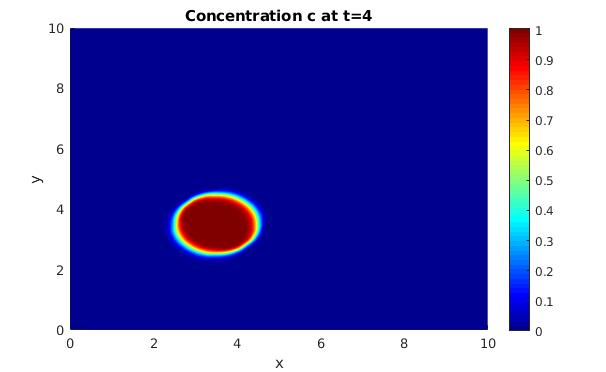}  
\includegraphics[scale=0.35]{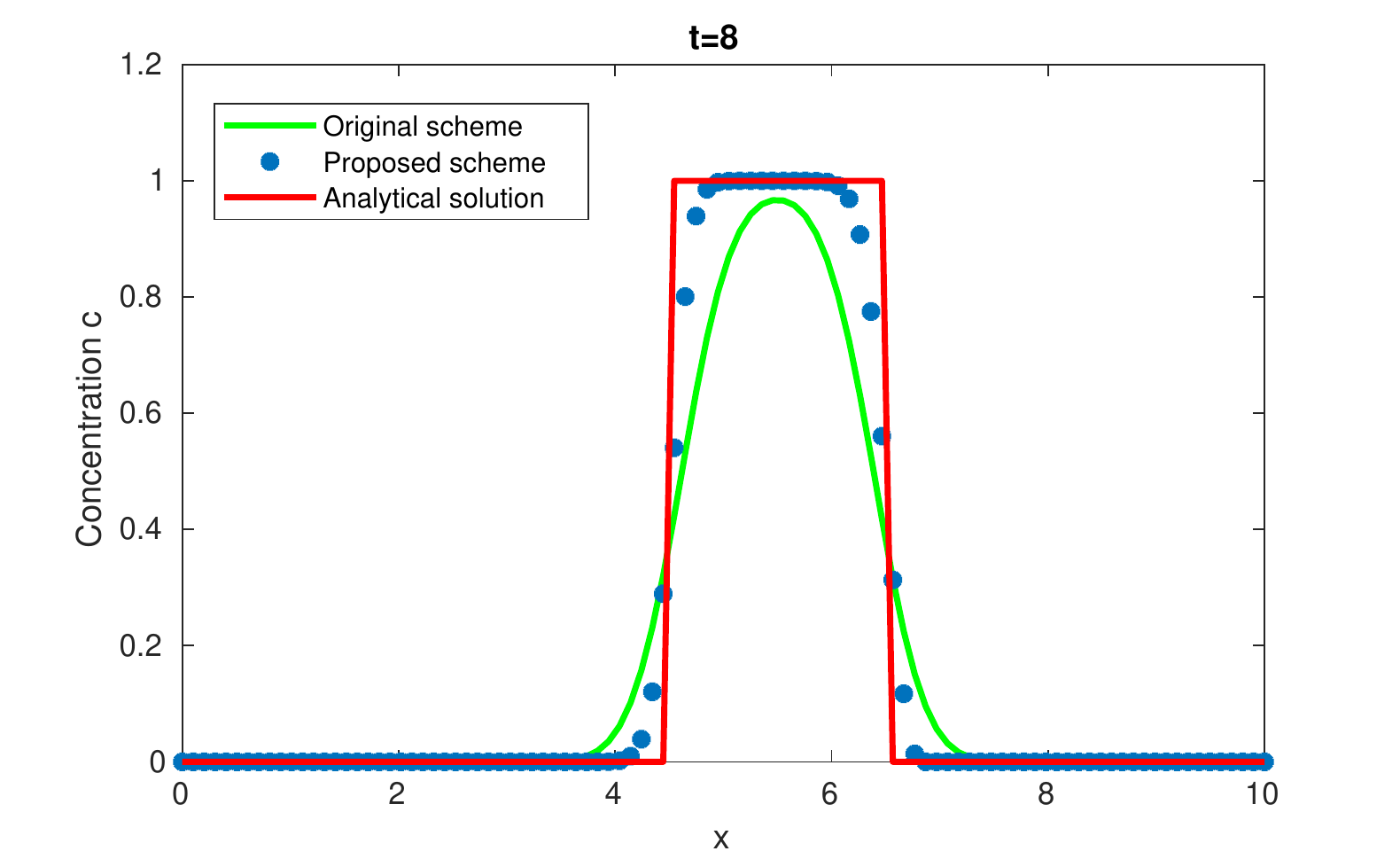} \quad \includegraphics[scale=0.25]{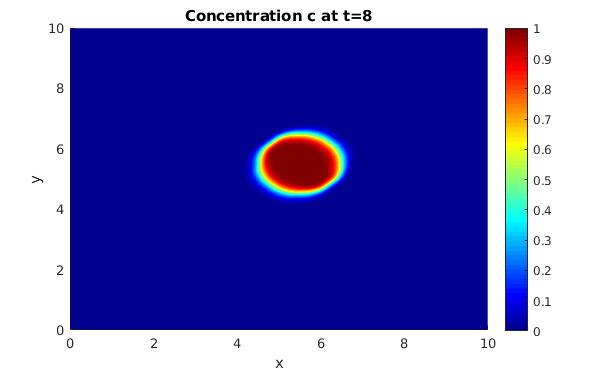} 
\includegraphics[scale=0.35]{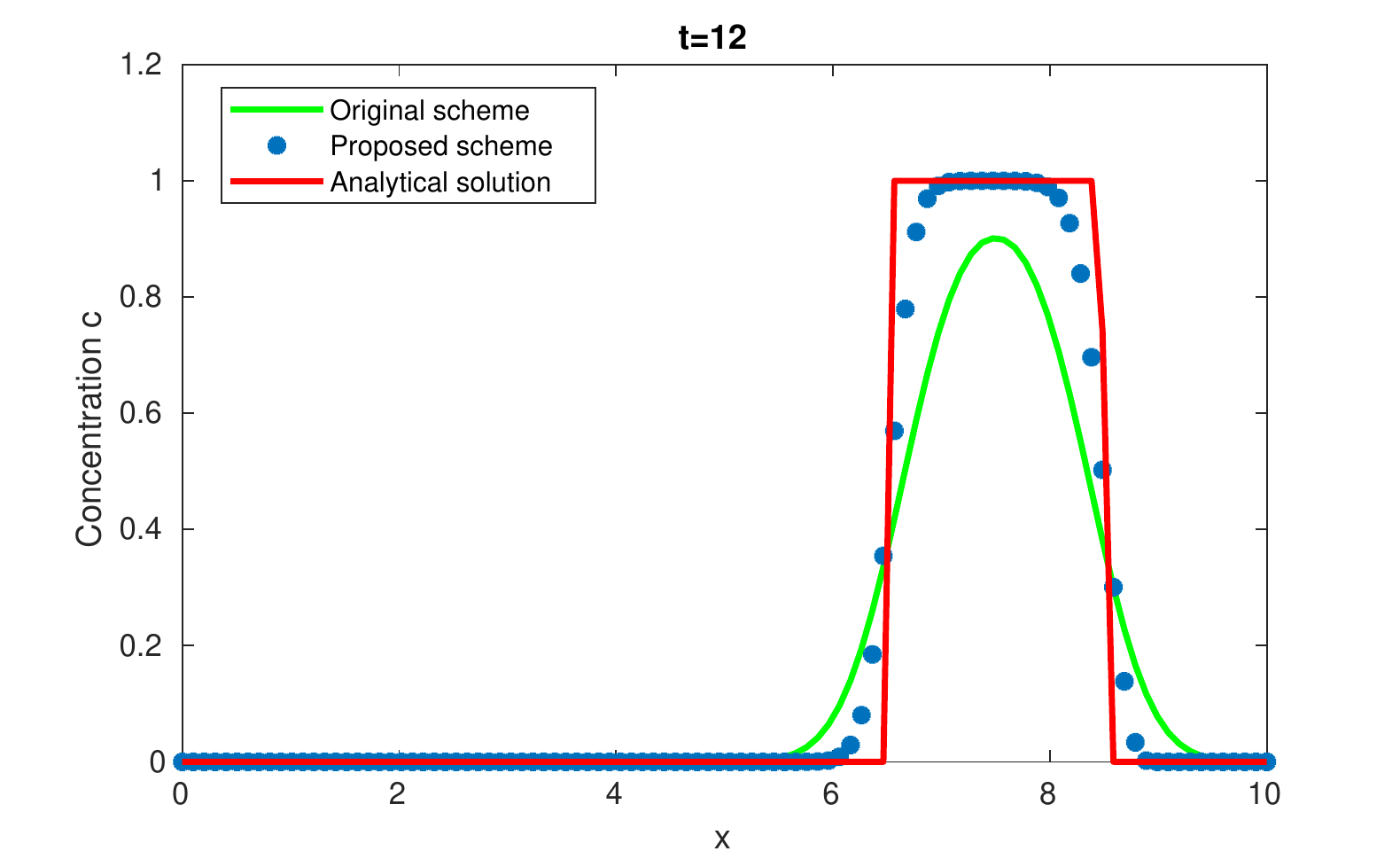} \quad \includegraphics[scale=0.25]{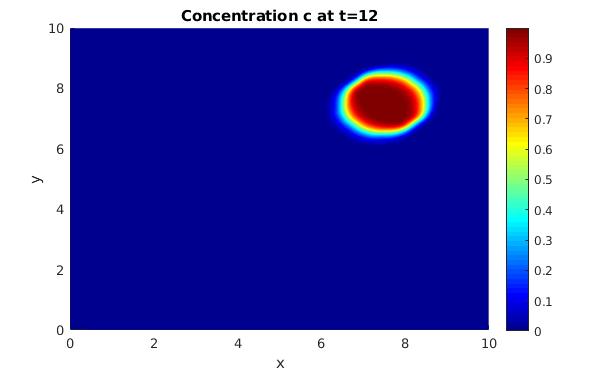} 
\caption{Modeling of a passive scalar concentration using the proposed method and the original scheme. Left : profiles of the concentration at time $t=1$, $t=4$, $t=8$ and $t=12$. Right: two-dimensional evolution of the computed scalar concentration using the proposed method.}
\label{Fig22}
\end{center}
\end{figure}
\newpage
\subsection{Example 3: Scalar diffusion process}\label{T3}
In this example we consider the diffusion phenomenon of a Gaussian distribution of the concentration of a pollutant in still water \citep{behzadi2018solution,vanzo2016pollutant}. We used the proposed numerical model to perform numerical simulations of the diffusion process of a pollutant 
and compare our results to available analytical solutions. The computational domain is $[-2,2]\times[-2,2]$, and the water surface is initially at rest over a flat bottom topography with water depth $h=0.01$. Our system is  discetized  using an average cell area $\left |\bar{D_{j}} \right |=7.292 \times 10^{-4}$. In the numerical 
test, the following analytical solution of the concentration \citep{behzadi2018solution,vanzo2016pollutant} is employed with outflow boundary conditions on all sides of the domain.
\begin{equation}
c(x,y,t)=\dfrac{\sigma^2}{4\gamma t+\sigma^2}\exp\left(\dfrac{-(x-x_0)^2}{4\gamma t+\sigma^2}+\dfrac{-(y-y_0)^2}{4\gamma t+\sigma^2}\right).
\end{equation}
The initial distribution of the concentration is a Gaussian centered at $ x_0=y_0=0$ with a standard deviation  $\sigma=0.1$ and diffusion coefficient $\gamma=0.01$. The numerical results of the proposed scheme are displayed in Figure \ref{Fig3}, where the computed  concentration is compared with the analytical solution at different times. Our results confirm that the predicted concentration profiles, which are obtained using the proposed numerical model, agree quite well with the analytical solution. \\To assess the accuracy of our results, we compute the following $L_2$-error for the concentration of the pollutant:
\begin{equation}
E=\sqrt{\dfrac{\sum_j \mid D_j \mid \mid {c}_{j}^{(n)}-{c}_{j}^{(a)} \mid^2 }{\sum_j \mid D_j \mid}},
\end{equation}
where ${c}_{j}^{(n)}$ and ${c}_{j}^{(a)}$ are the numerical and analytical solutions at the center of mass of the cell $D_j$, respectively. The computed solution by the proposed scheme at time $t = 15$ using $20201$ computational cells has an $L_2$-error $E=4.02\times  10^{-6}$ which confirms the accuracy of our results. 
\begin{figure}[h!]
\begin{center}
\includegraphics[scale=0.33]{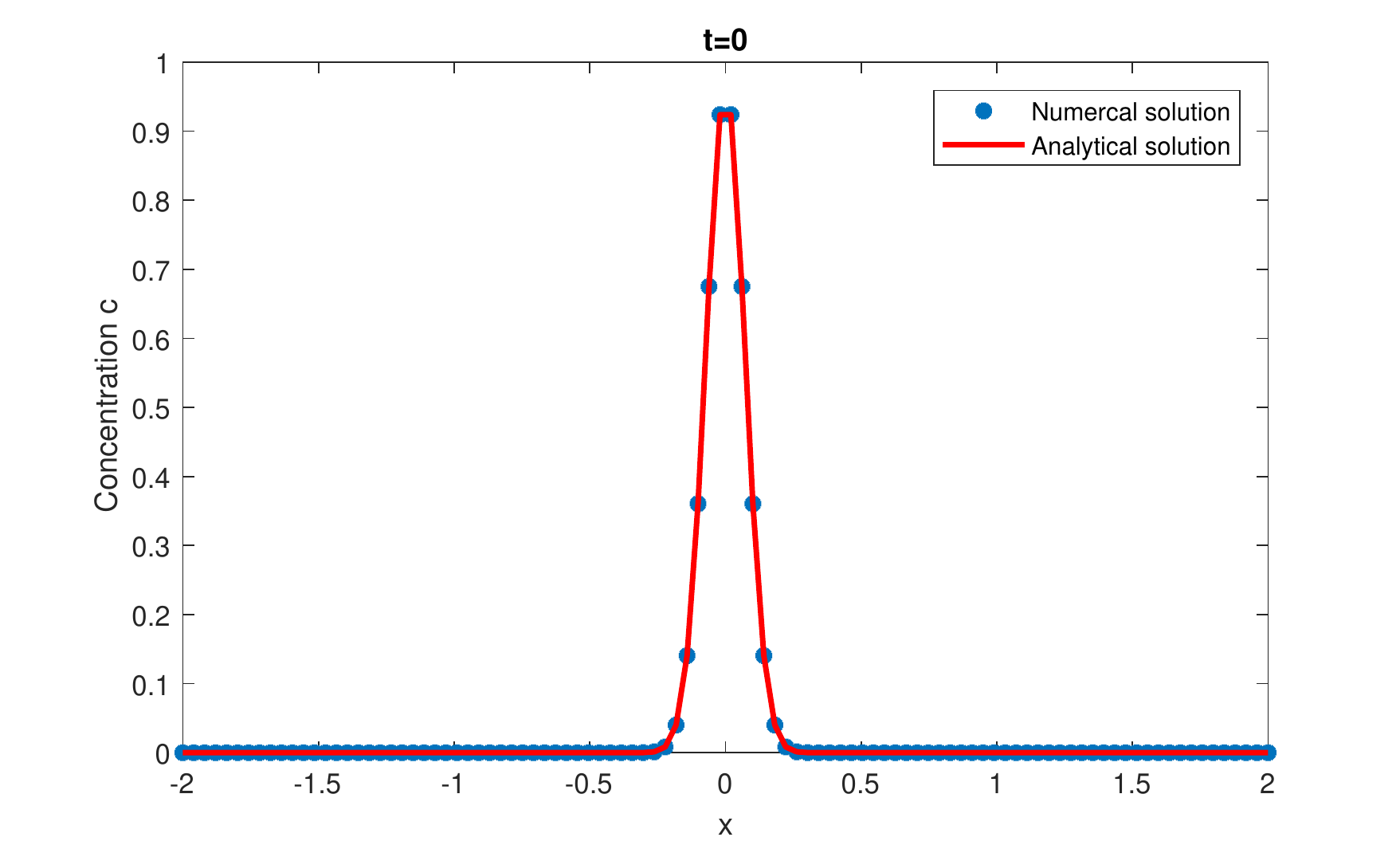} \quad \includegraphics[scale=0.22]{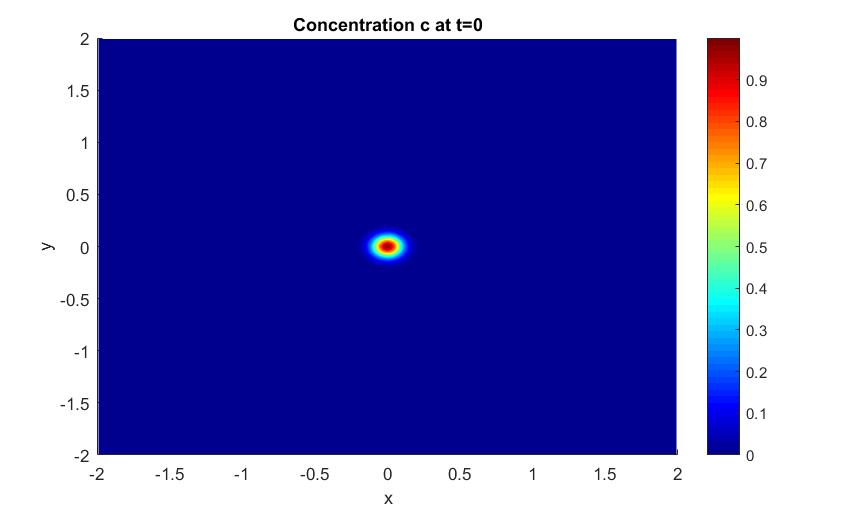} 
\includegraphics[scale=0.33]{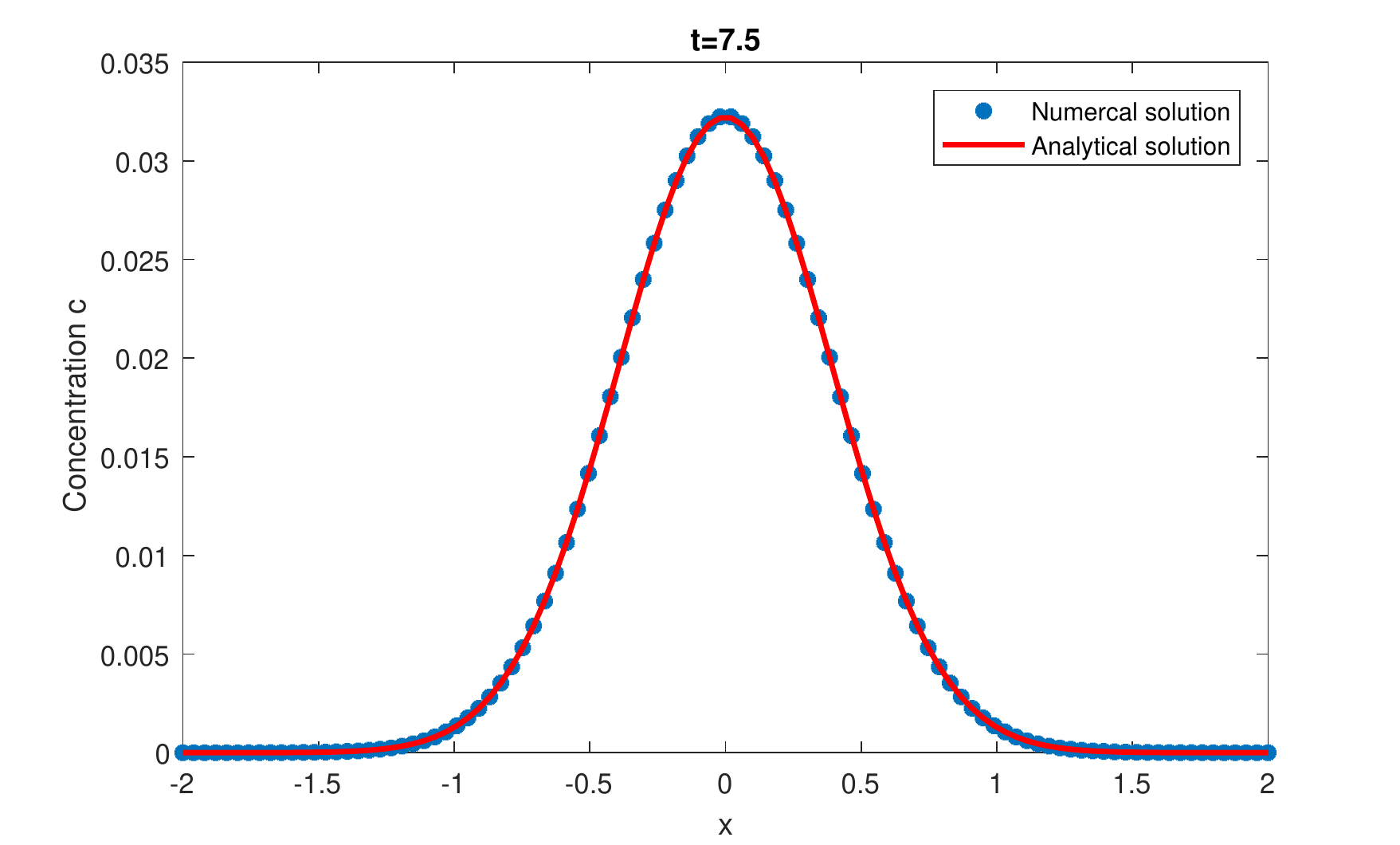} \quad \includegraphics[scale=0.22]{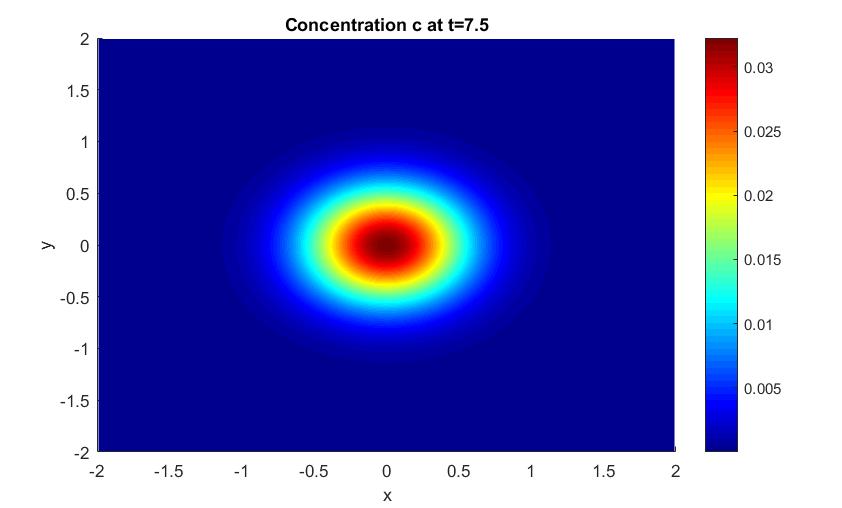} 
\includegraphics[scale=0.33]{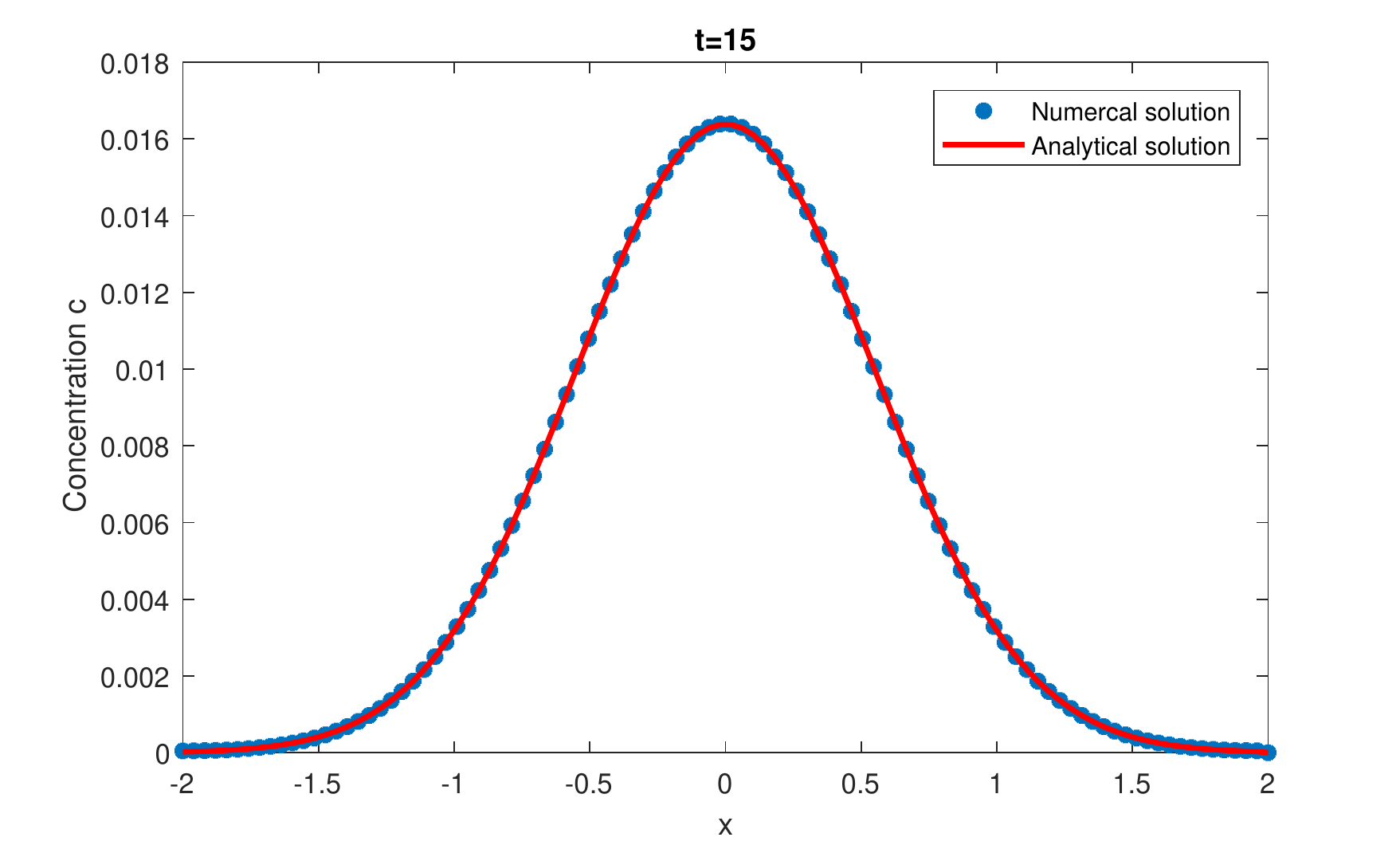} \quad \includegraphics[scale=0.22]{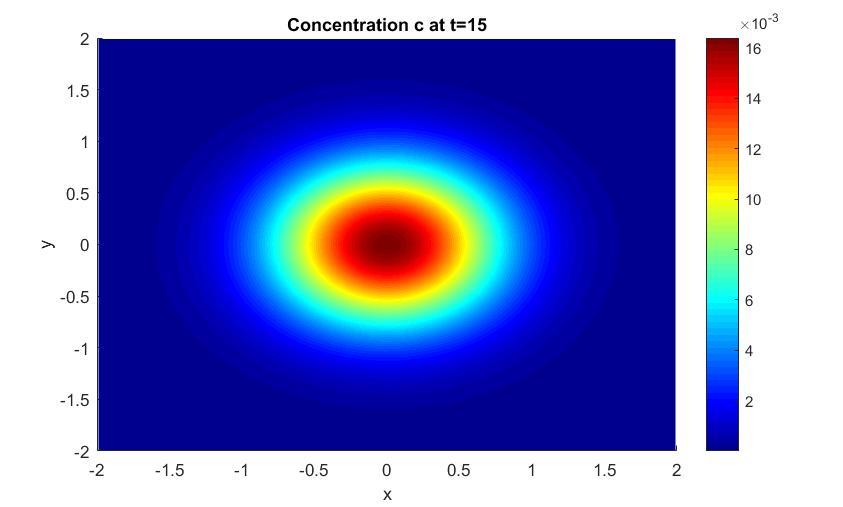} 
\caption{Modeling of the scalar diffusion process using the proposed method. The computed and analytical solutions of the scalar concentration at different times $t=0$, $t=7.5$ and $t=15$. }
\label{Fig3}
\end{center}
\end{figure}

\subsection{Example 4: Convection-diffusion process}\label{T4}
To further validate the capability of the proposed numerical scheme to predict the solute distribution in the presence of the diffusion term,  we perform numerical simulations for an advection-diffusion problem \citep{vanzo2016pollutant}. The numerical test is performed using the following analytical solution of the scalar concentration, given by \citep{vanzo2016pollutant,behzadi2018solution}:
\begin{equation}
c(x,y,t)=\dfrac{M_s}{4 \pi h \gamma t}\exp\left(\dfrac{-(x-ut-x_0)^2}{4\gamma t}+\dfrac{-(y-vt-y_0)^2}{4 \gamma t}\right),
\label{pd}
\end{equation}
where in our numerical test we used the diffusive coefficient $\gamma=0.01$, $M_s=0.1$ and $x_0=y_0=-0.45$. 

We start our numerical simulations using an initial condition for the concentration obtained by setting $t = 0.1$ in (\ref{pd}). We consider an uniform velocity field of the flow  $u=v=0.5$ and a constant water depth 
$h=1$ as initial condition. The computational domain is $[-1,1]\times[-1,1]$ with a flat and frictionless bed. The system is discretized  using an average cell area $\left |\bar{D_{j}} \right |=4.97 \times 10^{-5}$  and we used inflow and outflow conditions at the boundaries of the domain. The predicted concentration by using the proposed scheme is compared with the analytical solution in Figure \ref{FIG4} (left) at different times, while Figure \ref{FIG4} (right) presents the two-dimensional view of the computed solution.  The results of our numerical simulations show that the proposed scheme performs well in the prediction of the solution of the convection-diffusion process. The scalar concentration moves diagonally across the domain with the constant speed of the hydrodynamic field $u=v= 0.5$, and the peak level of the concentration decreases from 8.9 to 0.51, due to the presence of the physical diffusion term with accurate results compared to the analytical solution. The numerical solution at time $t = 1.5$ for the concentration of the pollutant, computed by the proposed method using $80401$ cells, has an $L_2$-error of $E=2 \times 10^{-3}$.
\begin{figure}[h!]
\begin{center}
\includegraphics[scale=0.33]{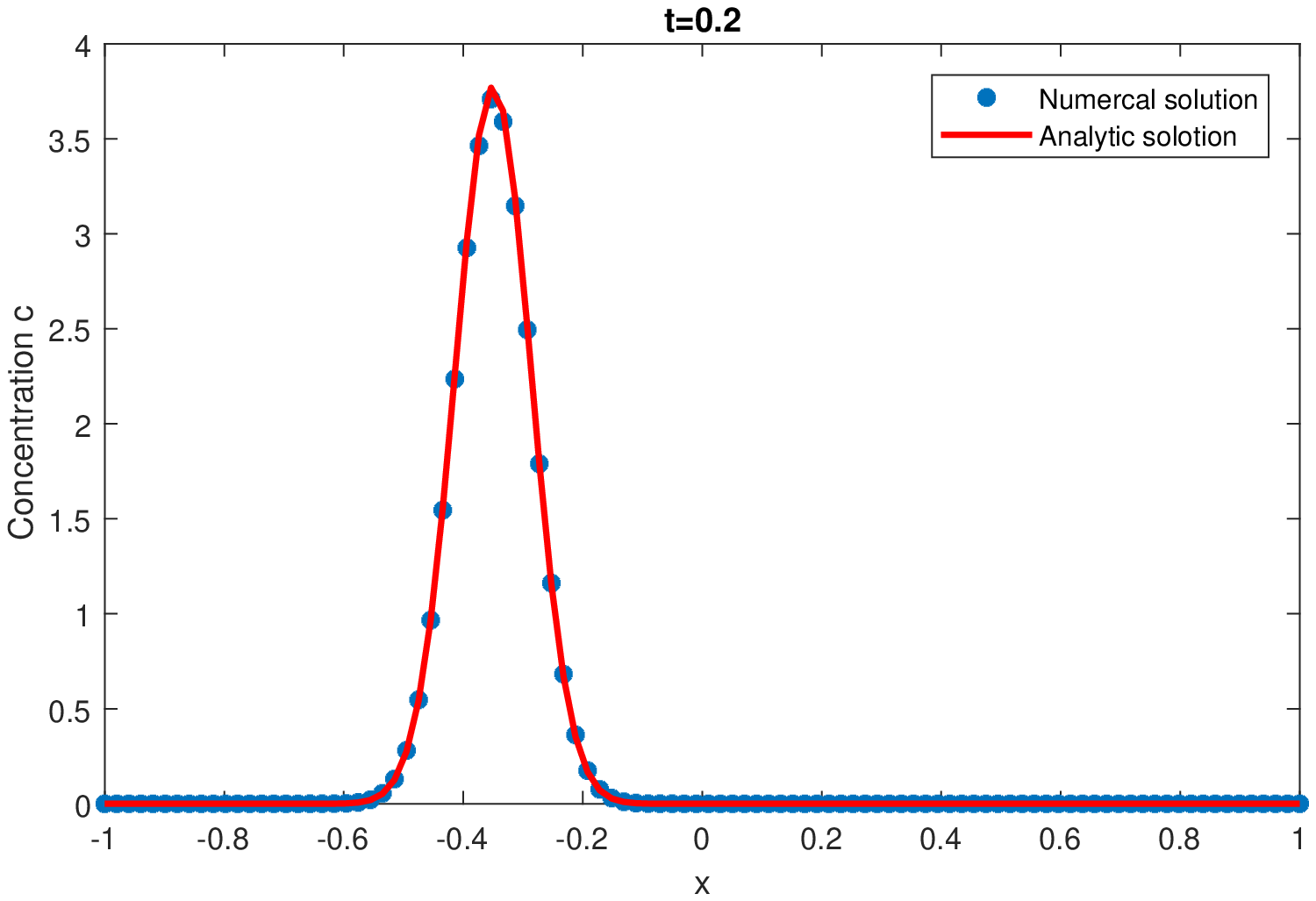} 
 \includegraphics[scale=0.22]{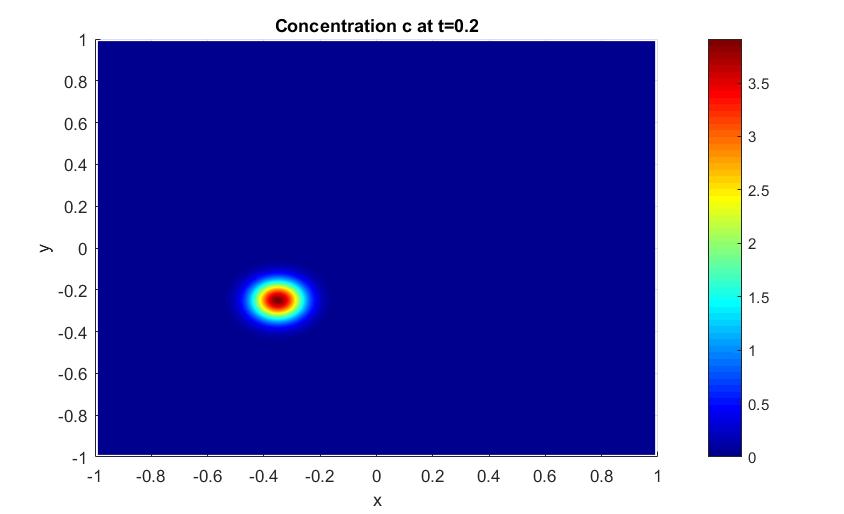} \quad
 
\includegraphics[scale=0.33]{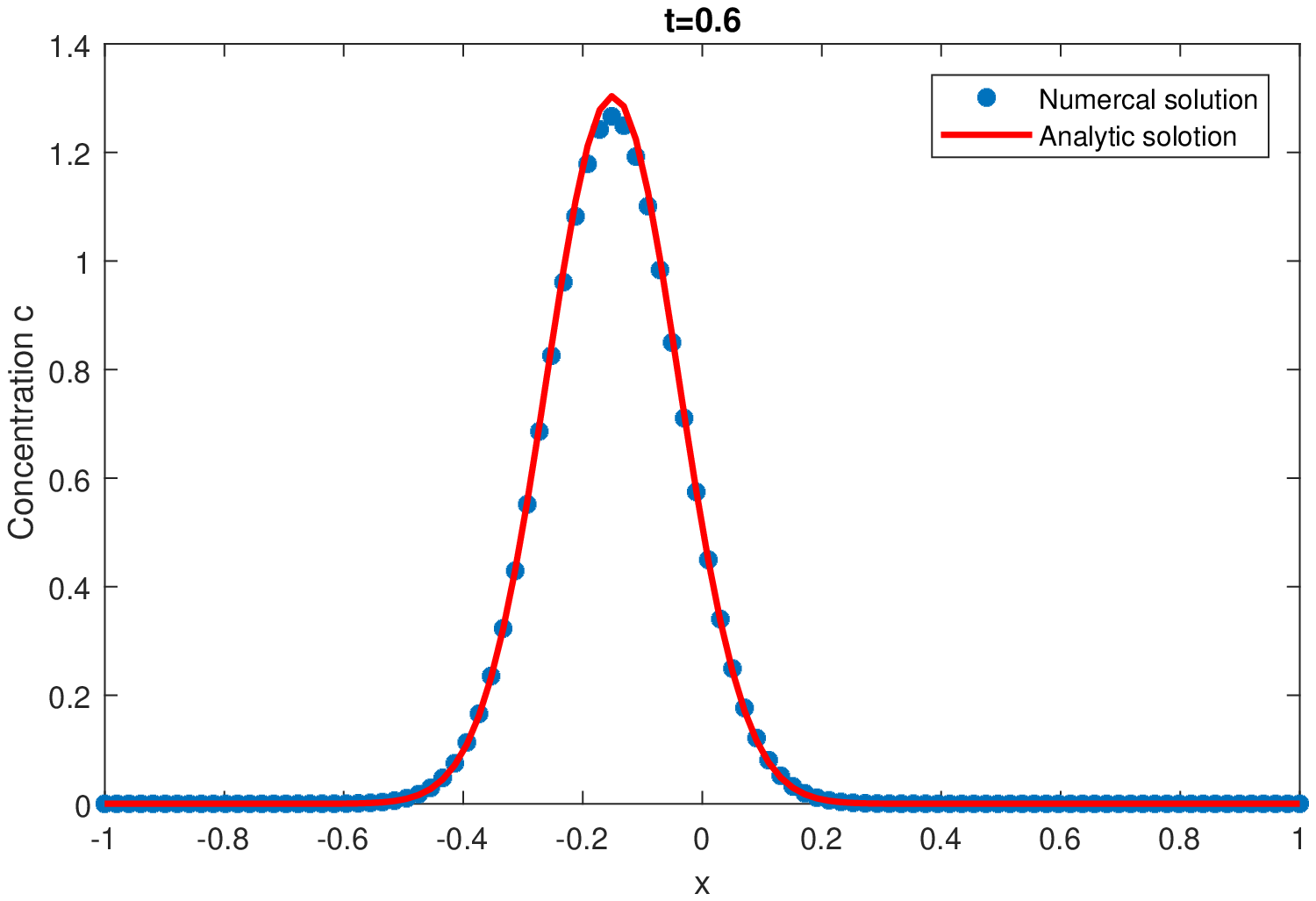} 
\includegraphics[scale=0.22]{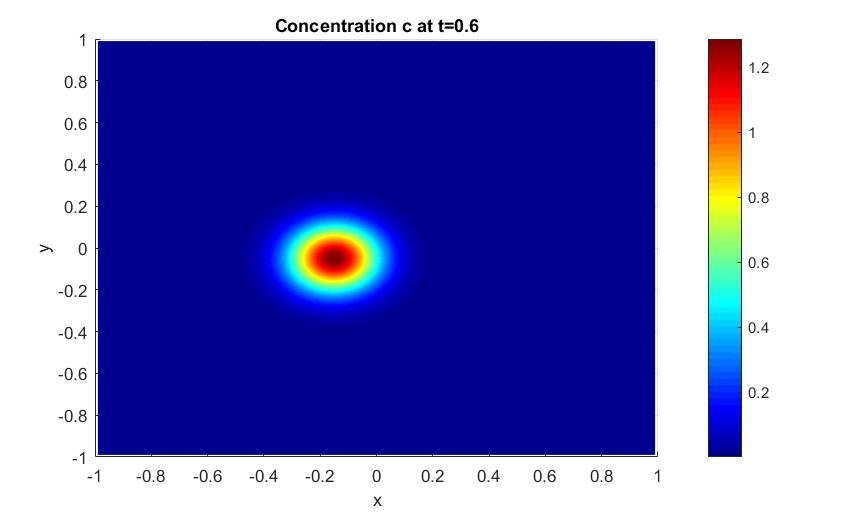} \quad

\includegraphics[scale=0.33]{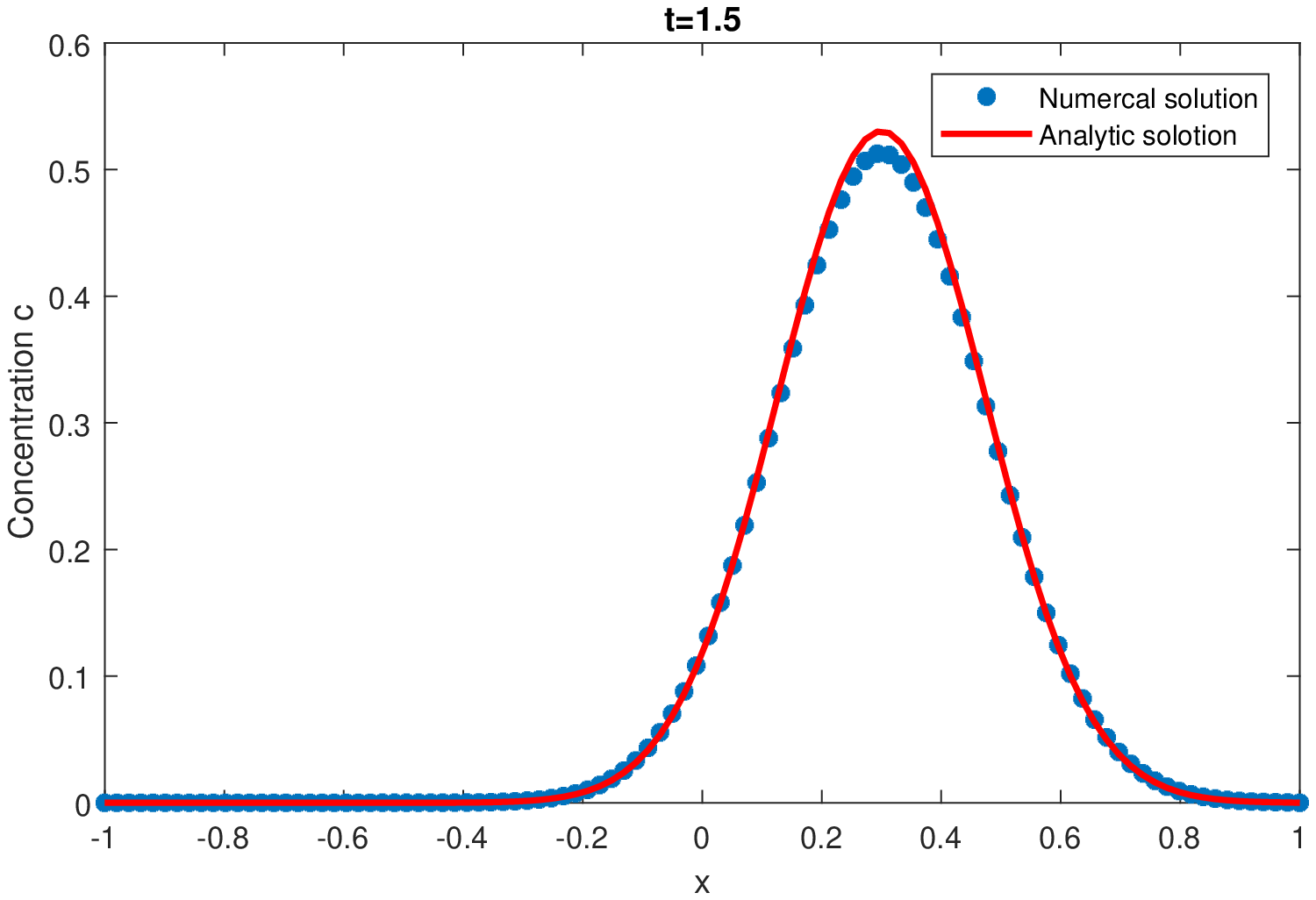} 
 \includegraphics[scale=0.22]{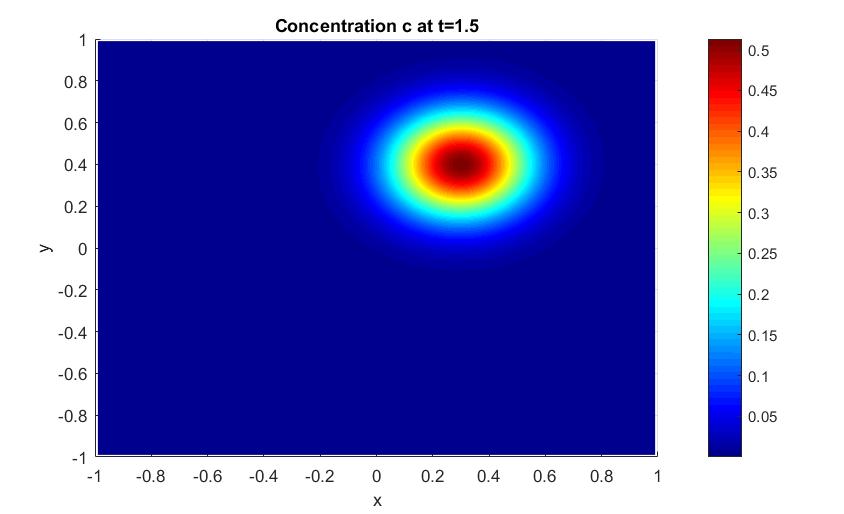} 
\caption{Modeling of the convection-diffusion process using the proposed scheme. Numerical and analytical solutions at times $t=0.2$, $t=0.6$ and $t=1.5$. }
\label{FIG4}
\end{center}
\end{figure}

\newpage
\subsection{Example 5: Pollutant transport over complex bottom topography}\label{T5}
Here we perform a numerical test which is widely used in previous studies \citep{begnudelli2006unstructured,behzadi2018solution,moukalled2016finite,vanzo2016pollutant,liang2010well} as a benchmark utilized to measure the performance of numerical models to simulate problems involving flows with scalar transport over wet/dry areas. In this numerical example, the computational domain $[0,75]\times[-15,15]$ and we consider a 2-D dam-break located at $x=16$ with still water in the initial condition in which the water depth $h=1.875$ is used upstream the dam, whereas downstream the dam the water depth is set to zero. The pollutant has an uniform concentration $c=1$ in the wet part for  $x\leq 10$ and $c=0$ in the remaining wet part. Hence the initial conditions are:
\begin{equation}
h(x,y,0)=\left\{ 
\begin{array}{ll}
1.875, & \text{if }\quad x\leq 16 , \\ 
0,& \quad\text{otherwise,} %
\end{array}%
\right. ,\quad c(x,y,0)=\left\{ 
\begin{array}{ll}
1, & \text{if }\quad x\leq 10 , \\ 
0,& \quad\text{otherwise,} %
\end{array}%
\right. \quad \bm{u}(x,y,0)=0.
\end{equation}

We consider the following variable topography with three humps, two of which are small and located respectively at $(x,y)=(30,6)$ and $(x,y)=(30,24)$, and the bigger one is located at $(x,y)=(47.5,15)$. 
\begin{equation}
\begin{aligned}
B(x,y)=\max & \left ( 0,1-\dfrac{1}{8}\sqrt{(x-30)^2+(y+9)^2},3-\dfrac{3}{10}\sqrt{(x-47.5)^2+y^2},\right.
\\ &\left. 1-\dfrac{1}{8}\sqrt{(x-30)^2+(y+9)^2}\right),
\end{aligned}
\end{equation}
and the Manning's coefficient is set to $n_f= 0.018$.

Our numerical simulations are performed using wall boundary conditions at all sides of the domain which is discretized using an average cell area $\left |\bar{D_{j}} \right |=2.98 \times 10^{-2}$. Figures \ref{Fig5} and \ref{Fig6} illustrate the evolution of computed solutions of the water 
depth (left) and the scalar concentration (right). The results of this numerical tests using the proposed numerical scheme are comparable with those of previous studies \citep{moukalled2016finite,behzadi2018solution,vanzo2016pollutant,begnudelli2006unstructured,liang2010well}). As shown in these figures at time $t=2$, the water reaches the two smaller humps, and the front of the zone with solute concentration is still located approximately at $x=10$. At time $t = 6$, the dam-break flow has passed and submerged the two smaller humps and water started to climb the big one, and the pollutant concentration had evolved by keeping an aligned front. At time $t=12$, we observe that the scalar concentration is affected by the flow and its shape is deformed while keeping a symmetrical distribution.
At $t = 18$, the water has reached the end of our domain and a small water wave is created which returns inside the domain. The flow start to stabilize with time and finally, at time $t=250$ steady state is almost achieved with velocities approaching zero and the three humps are partially wet (dry on the pick). At the same time, the distribution of the scalar concentration in the wet region has smooth profile. 
\begin{figure}
\begin{center}
\includegraphics[scale=0.3]{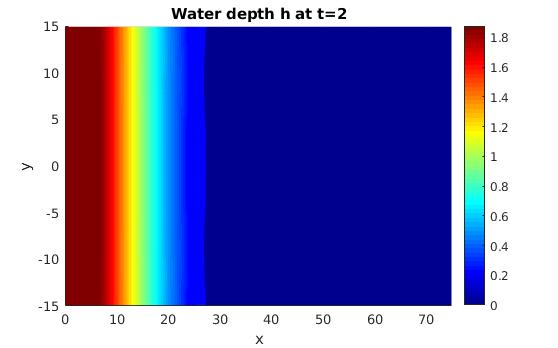} \quad \includegraphics[scale=0.3]{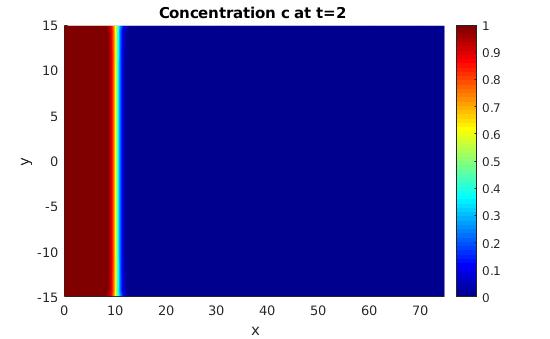} 

\includegraphics[scale=0.3]{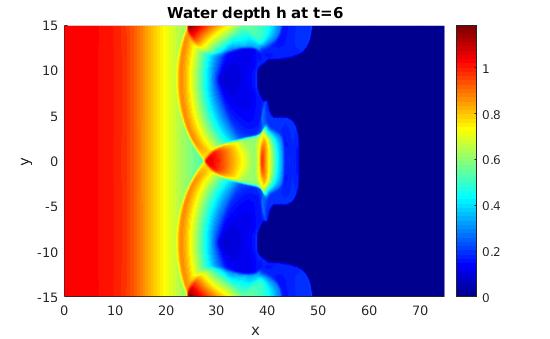} \quad \includegraphics[scale=0.3]{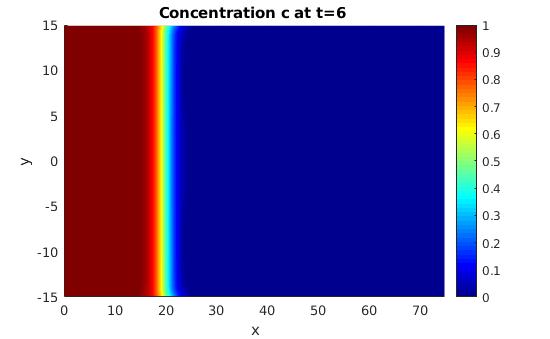}

\includegraphics[scale=0.3]{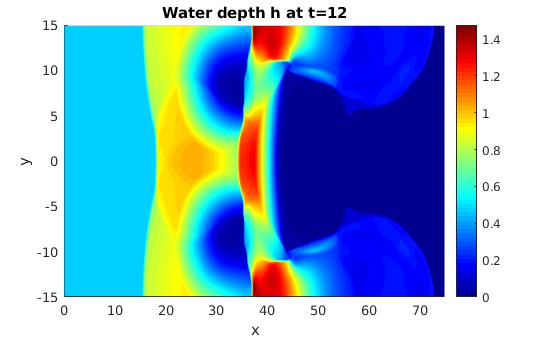} \quad \includegraphics[scale=0.3]{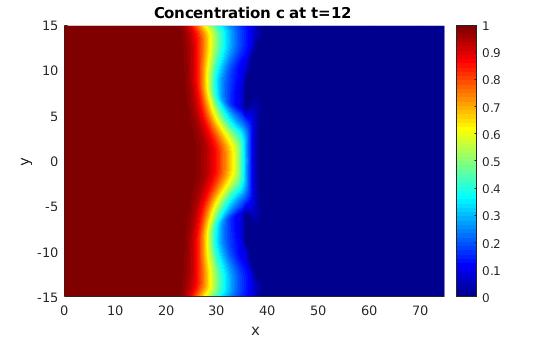} 

\includegraphics[scale=0.3]{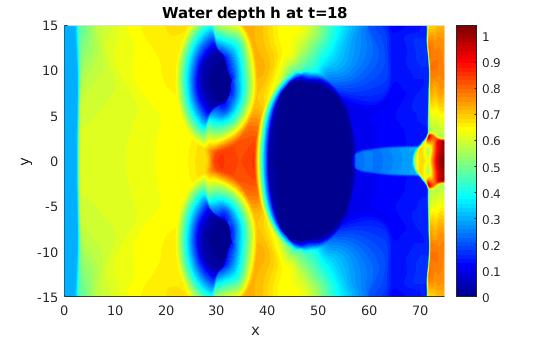} \quad \includegraphics[scale=0.3]{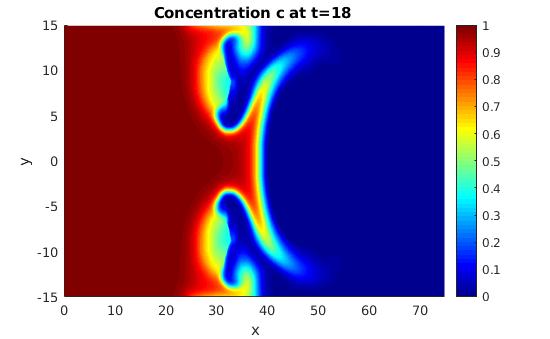} 
\caption{Distribution of water depth (left) and the concentration  (right) in a dam break over complex bottom topography at times $t=2$, $t=6$, $t=12$, and $t=18$.}
\label{Fig5}
\end{center}
\end{figure}
\begin{figure}[h!]
\begin{center}
\includegraphics[scale=0.3]{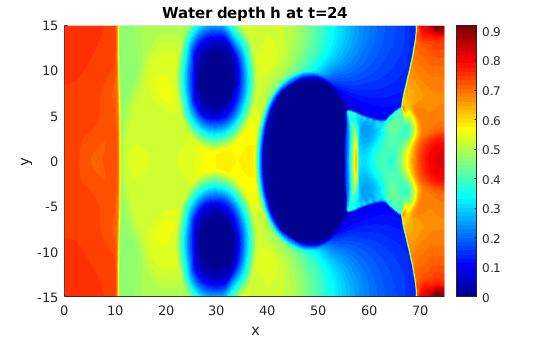} \quad \includegraphics[scale=0.3]{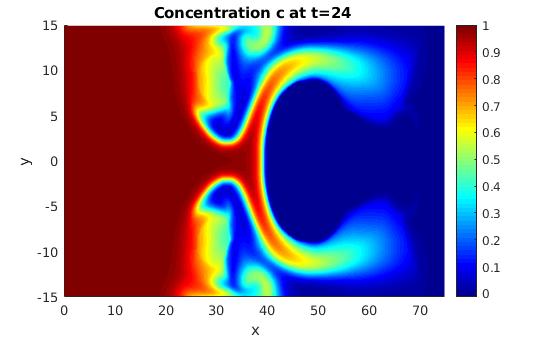} 

\includegraphics[scale=0.3]{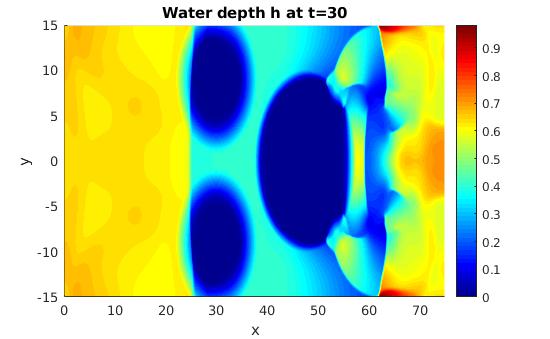} \quad \includegraphics[scale=0.3]{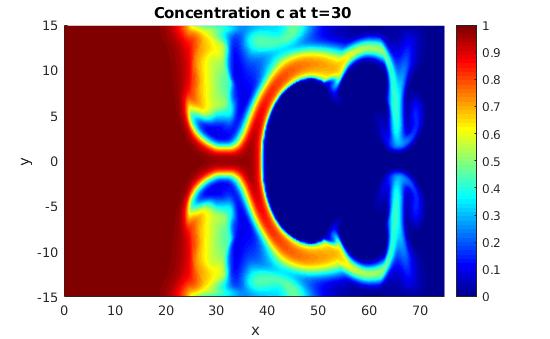} 

\includegraphics[scale=0.3]{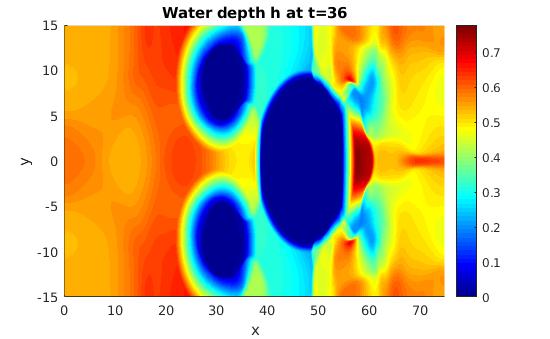} \quad \includegraphics[scale=0.3]{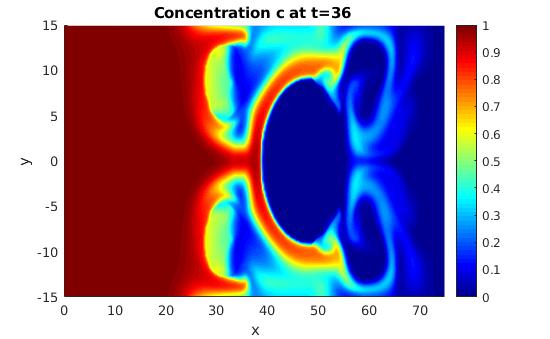} 

\includegraphics[scale=0.3]{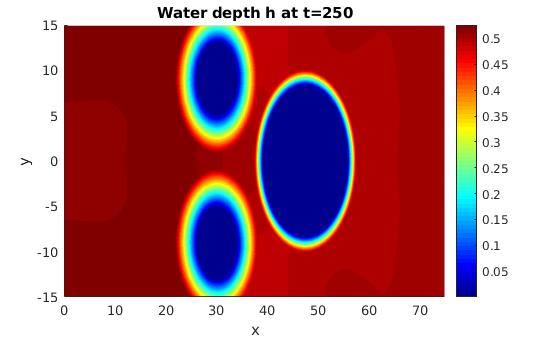} \quad \includegraphics[scale=0.3]{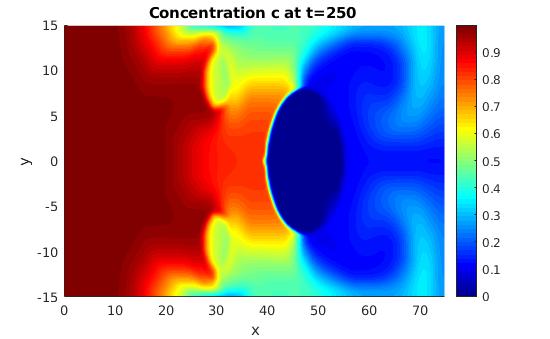} 
\caption{Distribution of water depth (left) and the concentration  (right) in a dam break over complex bottom topography at times  $t=24$, $t=30$, $t=36$, and $t=250$. }
\label{Fig6}
\end{center}
\end{figure}
\section{Concluding remarks}\label{S10}
In this study, we proposed a well-balanced positivity preserving finite volume scheme for modeling coupled systems of shallow water flows and scalar transport model over variable topography with diffusion term and source term due bottom friction effects.
Our approach is based on the semi-discrete formulation of the central-upwind scheme on cell-vertex grids \cite{beljadid2016well} where we developed new discretization techniques for the water surface elevation and the concentration of the pollutant. The proposed techniques have advantages in reducing numerical dissipation observed in the original scheme \cite{beljadid2016well} for the scalar concentration. The developed method preserves the steady state of a lake at rest and the positivity of the water depth. For the scalar concentration, the proposed scheme preserves the positivity and a perfect balance of the concentration 
where constant-concentration states are preserved in space and time for any hydrodynamic field of the flow over variable bottom topography in the absence of source terms in the scalar transport equation. We proved that the proposed scheme satisfies the discrete maximum-minimum principle for the scalar 
concentration. Our numerical simulations have confirmed the well-balanced and positivity preserving properties of the proposed scheme and the accuracy of our techniques in predicting the solutions of the coupled model of shallow water flow and solute transport.

\section*{\textbf{Acknowledgments}}
The second author acknowledges funding from UM6P/OCP Group. 
The third author was partially supported by the Innovative Training Networks (ITN) grant 642768 (ModCompShock) and by the Centre National de la Recherche Scientifique (CNRS). 

\bibliographystyle{elsarticle-num}

\bibliography{ref}

\begin{thebibliography}{10}

\bibitem{audusse2004fast}
E.~Audusse, F.~Bouchut, M.~O. Bristeau, R.~Klein, and B.~Perthame.
\newblock A fast and stable well-balanced scheme with hydrostatic
  reconstruction for shallow water flows.
\newblock {\em SIAM Journal on Scientific Computing}, 25(6):2050--2065, 2004.

\bibitem{begnudelli2006unstructured}
L.~Begnudelli and B.~F. Sanders.
\newblock Unstructured grid finite-volume algorithm for shallow-water flow and
  scalar transport with wetting and drying.
\newblock {\em Journal of hydraulic engineering}, 132(4):371--384, 2006.

\bibitem{behzadi2018solution}
F.~Behzadi, B.~Shamsaei, and J.~C. Newman.
\newblock Solution of fully-coupled shallow water equations and contaminant
  transport using a primitive-variable riemann method.
\newblock {\em Environmental Fluid Mechanics}, 18(2):515--535, 2018.

\bibitem{beljadid2017central}
A.~Beljadid and P.G. LeFloch.
\newblock A central-upwind geometry-preserving method for hyperbolic
  conservation laws on the sphere.
\newblock {\em Communications in Applied Mathematics and Computational
  Science}, 12(1):81--107, 2017.

\bibitem{beljadid2016well}
A.~Beljadid, A.~Mohammadian, and A.~Kurganov.
\newblock Well-balanced positivity preserving cell-vertex central-upwind scheme
  for shallow water flows.
\newblock {\em Computers \& Fluids}, 136:193--206, 2016.

\bibitem{beljadid2013unstructured}
A.~Beljadid, A.~Mohammadian, and H.~M. Qiblawey.
\newblock An unstructured finite volume method for large-scale shallow flows
  using the fourth-order adams scheme.
\newblock {\em Computers \& Fluids}, 88:579--589, 2013.

\bibitem{berthon2012efficient}
C.~Berthon and F.~Foucher.
\newblock Efficient well-balanced hydrostatic upwind schemes for shallow-water
  equations.
\newblock {\em Journal of Computational Physics}, 231(15):4993--5015, 2012.

\bibitem{bollermann2013well}
A.~Bollermann, G.~Chen, A.~Kurganov, and S.~Noelle.
\newblock A well-balanced reconstruction of wet/dry fronts for the shallow
  water equations.
\newblock {\em Journal of Scientific Computing}, 56(2):267--290, 2013.

\bibitem{bonev2018discontinuous}
B~Bonev, J~S Hesthaven, F~X Giraldo, and M~A Kopera.
\newblock Discontinuous galerkin scheme for the spherical shallow water
  equations with applications to tsunami modeling and prediction.
\newblock {\em Journal of Computational Physics}, 362:425--448, 2018.

\bibitem{brufau2003unsteady}
P.~Brufau and P.~Garc{\i}a-Navarro.
\newblock Unsteady free surface flow simulation over complex topography with a
  multidimensional upwind technique.
\newblock {\em Journal of Computational Physics}, 186(2):503--526, 2003.

\bibitem{bryson2011well}
S.~Bryson, Y.~Epshteyn, A.~Kurganov, and G.~Petrova.
\newblock Well-balanced positivity preserving central-upwind scheme on
  triangular grids for the saint-venant system.
\newblock {\em ESAIM: Mathematical Modelling and Numerical Analysis},
  45(3):423--446, 2011.

\bibitem{capilla2013new}
M.~T. Capilla and A.~Balaguer-Beser.
\newblock A new well-balanced non-oscillatory central scheme for the shallow
  water equations on rectangular meshes.
\newblock {\em Journal of Computational and Applied Mathematics}, 252:62--74,
  2013.

\bibitem{cea2010unstructured}
L.~Cea and M.~E. V{\'a}zquez-Cend{\'o}n.
\newblock Unstructured finite volume discretization of two-dimensional
  depth-averaged shallow water equations with porosity.
\newblock {\em International journal for numerical methods in fluids},
  63(8):903--930, 2010.

\bibitem{CeaVazquez}
L.~Cea and M.~E. V{\'a}zquez-Cend{\'o}n.
\newblock Unstructured finite volume discretisation of bed friction and
  convective flux in solute transport models linked to the shallow water
  equations.
\newblock {\em J. Comput. Phys.}, 231(8):3317--3339, 2012.

\bibitem{chertock2015well}
A.~Chertock, S.~Cui, A.~Kurganov, and T.~Wu.
\newblock Well-balanced positivity preserving central-upwind scheme for the
  shallow water system with friction terms.
\newblock {\em International Journal for numerical methods in fluids},
  78(6):355--383, 2015.

\bibitem{vcrnjaric2004improved}
N.~{\v{C}}rnjari{\'c}-{\v{Z}}ic, S.~Vukovi{\'c}, and L.~Sopta.
\newblock Improved non-staggered central nt schemes for balance laws with
  geometrical source terms.
\newblock {\em International journal for numerical methods in fluids},
  46(8):849--876, 2004.

\bibitem{de1871theorie}
B.~De~St~Venant.
\newblock Theorie du mouvement non-permanent des eaux avec application aux
  crues des rivers et a l'introduntion des marees dans leur lit.
\newblock {\em Academic de Sci. Comptes Redus}, 73(99):148--154, 1871.

\bibitem{delestre2013swashes}
O.~Delestre, C.~Lucas, P.~A. Ksinant, F.~Darboux, C.~Laguerre, T.~N.~T. Vo,
  F.~James, and S.~Cordier.
\newblock Swashes: a compilation of shallow water analytic solutions for
  hydraulic and environmental studies.
\newblock {\em International Journal for Numerical Methods in Fluids},
  72(3):269--300, 2013.

\bibitem{dong2020robust}
J~Dong.
\newblock A robust second-order surface reconstruction for shallow water flows
  with a discontinuous topography and a manning friction.
\newblock {\em Adv. Comput. Math.}, 46(2):35, 2020.

\bibitem{frolkovic1998maximum}
P.~Frolkovic.
\newblock Maximum principle and local mass balance for numerical solutions of
  transport equation coupled with variable density flow.
\newblock {\em Acta Mathematica Universitatis Comenianae}, 1(68):137--157,
  1998.

\bibitem{giraldo2002nodal}
F~X Giraldo, J~S Hesthaven, and T~Warburton.
\newblock Nodal high-order discontinuous galerkin methods for the spherical
  shallow water equations.
\newblock {\em Journal of Computational Physics}, 181(ARTICLE):499--525, 2002.

\bibitem{hernandez2016central}
G.~Hernandez-Duenas and A.~Beljadid.
\newblock A central-upwind scheme with artificial viscosity for shallow-water
  flows in channels.
\newblock {\em Advances in Water Resources}, 96:323--338, 2016.

\bibitem{kong2013high}
J.~Kong, P.~Xin, C.~J. Shen, Z.~Y. Song, and L.~Li.
\newblock A high-resolution method for the depth-integrated solute transport
  equation based on an unstructured mesh.
\newblock {\em Environmental modelling \& software}, 40:109--127, 2013.

\bibitem{kurganov2002central}
A.~Kurganov and D.~Levy.
\newblock Central-upwind schemes for the saint-venant system.
\newblock {\em ESAIM: Mathematical Modelling and Numerical Analysis},
  36(3):397--425, 2002.

\bibitem{kurganov2007reduction}
A.~Kurganov and C.~T. Lin.
\newblock On the reduction of numerical dissipation in central-upwind schemes.
\newblock {\em Commun. Comput. Phys}, 2(1):141--163, 2007.

\bibitem{kurganov2001semidiscrete}
A.~Kurganov, S.~Noelle, and G.~Petrova.
\newblock Semidiscrete central-upwind schemes for hyperbolic conservation laws
  and hamilton--jacobi equations.
\newblock {\em SIAM Journal on Scientific Computing}, 23(3):707--740, 2001.

\bibitem{kurganov2005central}
A.~Kurganov and G.~Petrova.
\newblock Central-upwind schemes on triangular grids for hyperbolic systems of
  conservation laws.
\newblock {\em Numerical Methods for Partial Differential Equations: An
  International Journal}, 21(3):536--552, 2005.

\bibitem{kurganov2007s}
A.~Kurganov and G.~Petrova.
\newblock A second-order well-balanced positivity preserving central-upwind
  scheme for the saint-venant system.
\newblock {\em Commun. Math. Sci.}, 5(1):133--160, 2007.

\bibitem{leveque1998balancing}
R.~J. LeVeque.
\newblock Balancing source terms and flux gradients in high-resolution godunov
  methods: the quasi-steady wave-propagation algorithm.
\newblock {\em Journal of computational physics}, 146(1):346--365, 1998.

\bibitem{leveque2002finite}
R.~J. LeVeque.
\newblock {\em Finite volume methods for hyperbolic problems}, volume~31.
\newblock Cambridge university press, 2002.

\bibitem{li2012fully}
S~Li and C~J Duffy.
\newblock Fully-coupled modeling of shallow water flow and pollutant transport
  on unstructured grids.
\newblock {\em Procedia Environmental Sciences}, 13:2098--2121, 2012.

\bibitem{liang2010well}
Q.~Liang.
\newblock A well-balanced and non-negative numerical scheme for solving the
  integrated shallow water and solute transport equations.
\newblock {\em Communications in Computational Physics}, 7(5):1049, 2010.

\bibitem{liu2017coupled}
X.~Liu and A.~Beljadid.
\newblock A coupled numerical model for water flow, sediment transport and bed
  erosion.
\newblock {\em Computers \& Fluids}, 154:273--284, 2017.

\bibitem{moukalled2016finite}
F.~Moukalled, L.~Mangani, and M.~Darwish.
\newblock The finite volume method.
\newblock In {\em The Finite Volume Method in Computational Fluid Dynamics},
  pages 103--135. Springer, 2016.

\bibitem{murillo2005coupling}
J~Murillo, J~Burguete, P~Brufau, and P~Garc{\'\i}a-Navarro.
\newblock Coupling between shallow water and solute flow equations: analysis
  and management of source terms in 2d.
\newblock {\em International journal for numerical methods in fluids},
  49(3):267--299, 2005.

\bibitem{murillo2008analysis}
J~Murillo, P~Garc{\'\i}a-Navarro, and J~Burguete.
\newblock Analysis of a second-order upwind method for the simulation of solute
  transport in 2d shallow water flow.
\newblock {\em International journal for numerical methods in fluids},
  56(6):661--686, 2008.

\bibitem{murillo2006conservative}
J~Murillo, P~Garc{\'\i}a-Navarro, J~Burguete, and P~Brufau.
\newblock A conservative 2d model of inundation flow with solute transport over
  dry bed.
\newblock {\em International Journal for Numerical Methods in Fluids},
  52(10):1059--1092, 2006.

\bibitem{perthame2001kinetic}
B.~Perthame and C.~Simeoni.
\newblock A kinetic scheme for the saint-venant system with a source term.
\newblock {\em Calcolo}, 38(4):201--231, 2001.

\bibitem{ricchiuto2009stabilized}
M.~Ricchiuto and A.~Bollermann.
\newblock Stabilized residual distribution for shallow water simulations.
\newblock {\em Journal of Computational Physics}, 228(4):1071--1115, 2009.

\bibitem{russo2002central}
G.~Russo.
\newblock Central schemes and systems of balance laws, in “hyperbolic partial
  differential equations.
\newblock {\em Theory, Numerics and Applications”, edited by Andreas Meister
  and Jens Struckmeier, Vieweg, G{\"o}ttingen}, 2002.

\bibitem{stelling1983construction}
G.~S. Stelling.
\newblock On the construction of computational methods for shallow water flow
  problems.
\newblock 1983.

\bibitem{vanzo2016pollutant}
D.~Vanzo, A.~Siviglia, and E.~F. Toro.
\newblock Pollutant transport by shallow water equations on unstructured
  meshes: hyperbolization of the model and numerical solution via a novel flux
  splitting scheme.
\newblock {\em Journal of Computational Physics}, 321:1--20, 2016.

\bibitem{vreugdenhil2013numerical}
C.~B. Vreugdenhil.
\newblock {\em Numerical methods for shallow-water flow}, volume~13.
\newblock Springer Science \& Business Media, 2013.

\bibitem{wu2018well}
G.~Wu, Z.~He, L.~Zhao, and G.~Liu.
\newblock A well-balanced positivity preserving two-dimensional shallow flow
  model with wetting and drying fronts over irregular topography.
\newblock {\em Journal of Hydrodynamics}, 30(4):618--631, 2018.

\end{thebibliography}

\end{document}